\RequirePackage{algpseudocode}

\documentclass[final]{siamart0516}

\title{%
    Incomplete Selected Inversion for Linear-Scaling Electronic Structure Calculations%
    \thanks{%
        This publication is based on Chapter 3 of the author's PhD thesis \cite{Ett19}. %
        The author gratefully acknowledges financial support by the Chancellor's International Scholarship at the University of Warwick. %
    }
}
\author{%
    Simon Etter%
    \thanks{Department of Mathematics, National University of Singapore (\email{ettersi@nus.edu.sg}).}
}

\headers{Incomplete Selected Inversion}{S.\ Etter}

\usepackage{autonum}
\usepackage{subfig}

\newsiamthm{conjecture}{Conjecture}
\newsiamthm{assumption}{Assumption}
\newsiamremark{remark}{Remark}
\newsiamremark{example}{Example}

\usepackage{amssymb}

\DeclareMathOperator{\real}{Re}

\DeclareMathOperator{\trace}{Tr}

\DeclareMathOperator*{\argmin}{arg\,min}

\Crefname{ALC@unique}{Line}{Lines}

\usepackage{xcolor}
\definecolor{cred}{RGB}{230, 0, 0}
\definecolor{cgreen}{RGB}{0, 127, 0}
\definecolor{cblue}{RGB}{0, 0, 200}
\definecolor{corange}{RGB}{255, 100, 0}
\definecolor{cyellow}{RGB}{255, 215, 0}
\definecolor{cpurple}{RGB}{150, 0, 150}
\definecolor{ccyan}{RGB}{0, 180, 235}

\usepackage{tikz}
\usetikzlibrary{calc}
\usetikzlibrary{backgrounds}

\tikzstyle{cvertex}=[solid, circle,draw=black,line width=1 pt, inner sep=2pt, fill=white, font=\footnotesize]
\tikzstyle{cedge}=[solid, draw=black,line width=1 pt]

\newcommand{\convexpath}[2]{
    [
        create hullcoords/.code={
            \global\edef\namelist{#1}
            \foreach [count=\counter] \nodename in \namelist {
                \global\edef\numberofnodes{\counter}
                \coordinate (hullcoord\counter) at (\nodename);
            }
            \coordinate (hullcoord0) at (hullcoord\numberofnodes);
            \pgfmathtruncatemacro\lastnumber{\numberofnodes+1}
            \coordinate (hullcoord\lastnumber) at (hullcoord1);
        },
        create hullcoords
    ]
    ($(hullcoord1)!#2!-90:(hullcoord0)$)
    \foreach [
        evaluate=\currentnode as \previousnode using int(\currentnode-1),
        evaluate=\currentnode as \nextnode using int(\currentnode+1)
    ] \currentnode in {1,...,\numberofnodes} {
        let \p1 = ($(hullcoord\currentnode) - (hullcoord\previousnode)$),
        \n1 = {atan2(\y1,\x1) + 90},
        \p2 = ($(hullcoord\nextnode) - (hullcoord\currentnode)$),
        \n2 = {atan2(\y2,\x2) + 90},
        \n{delta} = {Mod(\n2-\n1,360) - 360}
        in
        {arc [start angle=\n1, delta angle=\n{delta}, radius=#2]}
        -- ($(hullcoord\nextnode)!#2!-90:(hullcoord\currentnode)$)
    }
}

\begin{document}

\maketitle

\begin{abstract}
    Pole Expansion and Selected Inversion (PEXSI) is an efficient scheme for evaluating selected entries of functions of large sparse matrices as required e.g.\ in electronic structure algorithms.
    We show that the triangular factorizations computed by the PEXSI scheme exhibit a localization property similar to that of matrix functions, and we present a modified PEXSI algorithm which exploits this observation to achieve linear scaling.
    To the best of our knowledge, the resulting incomplete PEXSI (iPEXSI) algorithm is the first linear-scaling algorithm which scales provably better than cubically even in the absence of localization, and we hope that this will help to further lower the critical system size where linear-scaling algorithms begin to outperform the diagonalization algorithm.
\end{abstract}



\section{Introduction}

A key challenge in the numerical simulation of electronic structure models like density functional theory or tight binding is the fast evaluation of functions of sparse matrices.
Specifically, given a symmetric and sparse Hamiltonian matrix $H \in \mathbb{R}^{n \times n}$ which depends on the atomic coordinates $y_I \in \mathbb{R}^{d}$ with $I \in \{1, \ldots, N_\mathrm{atoms}\}$, the task is to compute quantities such as
the total energy $E_\mathrm{total}$,
the force $F_I = -\frac{\partial E_\mathrm{total}}{\partial y_I}$ on atom $I \in \{1, \ldots, N\}$,
the electronic density $\rho_i$ at site $i \in \{1, \ldots, n\}$ or
the total number of electrons $N_\mathrm{electrons}$,
which are given, respectively, by
\begin{equation}\label{eqn:qoi_fun}
    \begin{aligned}
        E_\mathrm{total} &= \trace\bigl(H \, f_{\beta,E_F}(H)\bigr)
        ,\\
        F_I &= -\trace\Bigl(\bigl(f_{\beta,E_F}(H) + H \, f'_{\beta,E_F}(H)\bigr) \, \tfrac{\partial H}{\partial y_I}\Bigr)
        ,\\
        \rho_i &= f_{\beta,E_F}(H)_{ii}
        ,\\
        N_\mathrm{electrons} &= \trace\bigl(f_{\beta,E_F}(H)\bigr)
    \end{aligned}
\end{equation}
where
\[
f_{\beta,E_F}(E) := \frac{1}{1 + \exp\bigl(\beta \, (E - E_F)\bigr)}
\]
denotes the Fermi-Dirac function with inverse temperature $\beta > 0$ and Fermi energy $E_F \in \mathbb{R}$; see e.g.\ \cite{Goe99,Kax03,SCS10,BM12}.

\subsection{Electronic structure algorithms}\label{ssec:algorithms}

A simple algorithm for evaluating the matrix functions in \cref{eqn:qoi_fun} is to compute the $n$ eigenpairs $E_k \in \mathbb{R}$, $\psi_k \in \mathbb{R}^n$ of $H$ and to evaluate the quantities of interest using the formulae
\begin{equation}\label{eqn:qoi_eigen}
    \begin{aligned}
        E_\mathrm{total} &= \sum_{k = 1}^n E_k \, f_{\beta,E_F}(E_k)
        ,\\
        F_I &= - \sum_{k = 1}^n \bigl(f_{\beta,E_F}(E_k) + E_k \, f'_{\beta,E_F}(E_k)\bigr) \, \psi_k^T \tfrac{\partial H}{\partial y_I} \psi_k
        ,\\
        \rho_i &= f_{\beta,E_F}(E_k)_{ii} \, (\psi_k)_i^2
        ,\\
        N_\mathrm{electrons} &= \sum_{k = 1}^n f_{\beta,E_F}(E_k)
        .
    \end{aligned}
\end{equation}
This approach is known as the {diagonalization algorithm}.
Its main drawback is that computing all eigenpairs requires $\mathcal{O}(n^3)$ floating-point operations which is unaffordably expensive for many matrix sizes $n$ of scientific interest.

Cubic scaling can be avoided by exploiting that the density matrix $f_{\beta,E_F}(H)$ is exponentially localized for insulators and metals at finite temperature, i.e.\ the magnitude of the entries of $f_{\beta,E_F}(H)$ decay exponentially as we move away from the nonzero entries of $H$ if either the spectrum of $H$ contains a gap of width independent of $n$ around the Fermi energy $E_F$ (insulator case), or if we have $\beta < \infty$ (finite temperature case) \cite{Koh96,BBR13}.
This localization implies that even though $f_{\beta,E_F}(H)$ has $\mathcal{O}(n^2)$ nonzero entries, only $\mathcal{O}(n)$ of these entries are numerically significant and the quantities of interest can be computed in $\mathcal{O}(n)$ effort by replacing each occurrence of the exact density matrix $f_{\beta,E_F}(H)$ in \cref{eqn:qoi_fun} with a band-limited approximation $\tilde \Gamma \approx f_{\beta,E_F}(H)$ which can be determined in various ways, see the review articles \cite{Goe99,BM12}.
The resulting {linear-scaling algorithms} significantly extend the range of system sizes $n$ amenable to numerical simulation, but they provide no speedup over the diagonalization algorithm on small- to medium-sized systems due to a large prefactor in the $\mathcal{O}(n)$ cost estimate, and their performance deteriorates in the limit of metals at low temperatures due to vanishing localization.

\begin{table}
    \centering
    \def\TT{\rule{0pt}{2.6ex}}       
    \def\BB{\rule[-1.2ex]{0pt}{0pt}} 
    \begin{tabular}{| l | c | c | c |}
        \hline
        & Runtime & Memory & Example system \TT\BB\\
        \hline\hline
        $d = 1$ & $\mathcal{O}\big(n\big)$ & $\mathcal{O}\big(n\big)$ & Nanotubes \TT\BB\\\hline
        $d = 2$ & $\mathcal{O}\big(n^{3/2}\big)$ & $\mathcal{O}\big(n \, \log(n)\big)$ & Monolayers \TT\BB\\\hline
        $d = 3$ & $\mathcal{O}\big(n^{2}\big)$ & $\mathcal{O}\big(n^{4/3}\big)$ & Bulk solids \TT\BB\\
        \hline
    \end{tabular}
    \caption{
    Runtime and memory requirements of the selected inversion algorithm as a function of the effective dimension $d$ of the atomic system.
    References and some discussion regarding this result can be found in \cref{ssec:sparse_factorization}.
    }
    \label{tbl:selinv_costs}
\end{table}

An algorithm which scales strictly better than $\mathcal{O}(n^3)$ regardless of localization has recently been proposed in \cite{LCYH13}.
In abstract terms, this algorithm evaluates a matrix function $f(H)$ via a rational approximation in pole-expanded form,
\begin{equation}\label{eqn:rational_approximation}
    f(H)
    \approx
    r(H) :=
    \sum_{k = 1}^q w_k \, (H - z_kI)^{-1}
\end{equation}
where the weights $w_k \in \mathbb{C}$ and poles $z_k \in \mathbb{C}$ are chosen such that $r(E) \approx f(E)$ on the spectrum of $H$.
Applied to the matrix functions in \cref{eqn:qoi_fun}, this yields the formulae
\begin{equation}\label{eqn:qoi_rational}
    \begin{aligned}
        E_\mathrm{total} &\approx \sum_{k = 1}^q w_k^{(E)} \, \trace\bigl((H - z_k^{(E)} I)^{-1}\bigr)
        ,\\
        F_I &\approx -\sum_{k = 1}^q w_k^{(F)} \, \trace\Bigl( (H - z_k^{(F)} I)^{-1} \, \tfrac{\partial H}{\partial y_I} \Bigr)
        ,\\
        \rho_i &\approx \sum_{k = 1}^q w_k^{(f)} \, \bigl((H - z_k^{(f)} I)^{-1}\bigr)_{ii}
        ,\\
        N_\mathrm{electrons} &\approx \sum_{k = 1}^q w_k^{(f)} \, \trace\bigl((H - z_k^{(f)} I)^{-1}\bigr)
    \end{aligned}
\end{equation}
which replace the problem of evaluating the complicated matrix functions in \cref{eqn:qoi_fun} with that of evaluating shifted inverses $(H - z_kI)^{-1}$.
Moreover, only the diagonal of $(H-z_kI)^{-1}$ is required for evaluating $E_\mathrm{total}$, $\rho_i$ and $N_\mathrm{electrons}$, and only the nonzero entries corresponding to nonzeros in $\frac{\partial H}{\partial y_I}$ are required for $F_I$.
Both of these sets of entries can be computed efficiently using the selected inversion algorithm from \cite{ET75} with runtime and memory requirements as summarized in \cref{tbl:selinv_costs}.
We conclude from this table that as promised at the beginning of this paragraph, the {Pole Expansion and Selected Inversion} (PEXSI) algorithm implied by \cref{eqn:qoi_rational} scales strictly better than cubically in all dimensions regardless of localization, but it scales worse than linearly in two and three dimensions which puts this scheme at a disadvantage compared to linear-scaling methods.

\subsection{Contributions}\label{ssec:contributions}

After a brief review of the selected inversion algorithm in \cref{sec:selinv}, we will show in \cref{sec:localization} that the triangular factorizations computed as part of this algorithm exhibit a localization property similar to that of the density matrix $f_{\beta,E_F}(H)$.
This suggests that the PEXSI scheme can be turned into a linear-scaling method by restricting the selected inversion algorithm to the $\mathcal{O}(n)$ entries of non-negligible magnitude, and in \cref{sec:iselinv} we will propose and discuss in detail such an incomplete selected inversion algorithm.
\Cref{sec:experiments} will present numerical experiments which demonstrate the convergence and linear scaling of the proposed algorithm, and finally \cref{sec:conclusion} will compare our iPEXSI scheme against other linear-scaling algorithms and discuss its parallel implementation.

\subsection{Related work}

The triangular factorization part of the incomplete selected inversion algorithm presented in \cref{sec:iselinv} is exactly the symmetric version of the incomplete LU factorization commonly used as a preconditioner in iterative methods for linear systems, see e.g. \cite[\S10.3]{Saa03}.
Our analysis sheds a new light on this well-known algorithm which may also find applications outside of the context of electronic structure algorithms.

Three algorithms for computing the rational approximations required by \cref{eqn:qoi_rational} have been proposed in the literature: best rational approximations to the Fermi-Dirac function have been determined in \cite{Mou16}, approximation based on discretized contour integrals has been presented in \cite{LLYE09}, and a rational interpolation scheme has been proposed in \cite{Ett19}.
We expect that the approximations from \cite{Mou16} are the best choice for most applications since they deliver the best possible accuracy for a given number of poles $q$.
A distributed-memory parallel implementation of the selected inversion algorithm has been developed in \cite{JLY16}.

\section{Review of exact selected inversion}\label{sec:selinv}

Selected inversion applied to a matrix $A$ consists in first computing a triangular factorization of $A$ and then inferring the values $A^{-1}(i,j)$ from this factorization.
This section will introduce the appropriate triangular factorization in \cref{ssec:factorization}, recall some key definitions and results from the theory of sparse factorizations in \cref{ssec:sparse_factorization}, and finally describe how to compute selected entries of the inverse in \cref{ssec:selinv}.

\subsection{$LDL^T$ factorization}\label{ssec:factorization}

Throughout this paper, we assume that the Hamiltonian $H$ is a real and sparse matrix, i.e.\ we assume that the underlying partial differential equation is discretized using a real and spatially localized basis like atomic orbitals or finite elements, and we exclude plane-wave discretizations.
The matrices $A := H - z_kI$ passed to the selected inversion algorithm are therefore complex symmetric, i.e.\ they have complex entries $A(i,j) \in \mathbb{C}$ but satisfy $A(i,j) = A(j,i)$ rather than $A(i,j) = \overline{A(j,i)}$ due to the complex shifts $z_k \in \mathbb{C}$.
The appropriate triangular factorization for such matrices is the $LDL^T$ factorization introduced in the following theorem.

\begin{theorem}[{{\cite[Theorem 3.2.1]{GV96}}}]\label{thm:factorization_existence}
    Let $A \in \mathbb{C}^{n \times n}$ be a symmetric matrix such that all the leading submatrices $A(\ell,\ell)$ with $\ell = \{1, \ldots, i\}$ and $i$ ranging from $1$ to $n$ are invertible. Then, there exist matrices $L,D \in \mathbb{C}^{n \times n}$ such that $L$ is lower-triangular with unit diagonal, $D$ is diagonal and $A = LDL^T$.
\end{theorem}

\begin{definition}\label{rem:selinv_notation}
    We use the following notation throughout this section.
    \begin{itemize}
        \item $A \in \mathbb{C}^{n \times n}$ denotes a symmetric matrix satisfying the conditions of \cref{thm:factorization_existence}, and we denote by $L,D \in \mathbb{C}^{n \times n}$ the $LDL^T$ factors of $A$.
        \item Unless specified otherwise, $i,j \in \{1, \ldots, n\}$ refer to an entry in the lower triangle ($i \geq j$), and we set $\ell = \{1, \ldots, j-1\}$, $\bar \ell := \ell \cup \{j\}$, $r := \{j+1, \ldots, n\}$ and $\bar r := r \cup \{j\}$.
        \item $\mathrm{nz}(A)$ and $\mathrm{fnz}(A)$ denote the set of all nonzeros indices in $A$ and its $LDL^T$ factorization, respectively, i.e.\
        \begin{align}
            \mathrm{nz}(A)
            &:=
            \bigl\{
            (i,j) \in \{1, \ldots, n\}^2
            \mid
            A(i,j) \neq 0
            \bigr\}
            ,\\
            \mathrm{fnz}(H)
            &:=
            \bigl\{
            (i,j) \in \{1, \ldots, n\}^2
            \mid
            L(i,j) \neq 0
            \text{ or }
            L(j,i) \neq 0
            \bigr\}
            .
        \end{align}
    \end{itemize}
\end{definition}

The $LDL^T$ factorization of a given matrix $A$ may be computed using the well-known Gaussian elimination algorithm (see e.g.\ \cite[\S 3.2]{GV96}) which we will derive from the following result.\footnote{
Both \cref{thm:factorization_recursive,thm:factorization} were derived independently by the author, but given the importance of triangular factorizations and the simplicity of our formulae, we assume that similar statements have appeared previously in the literature.
}

\begin{theorem}\label{thm:factorization_recursive}
    In the notation of \cref{rem:selinv_notation}, {we have} that
    \begin{equation}\label{eqn:factorization_recursive}
        L(i,j) \, D(j,j) = A(i,j) - L(i,\ell) \, D(\ell,\ell) \, L^T(\ell,j)
        .
    \end{equation}
\end{theorem}

We observe that the right-hand side of \cref{eqn:factorization_recursive} depends only on entries $L(i,k)$, $D(k,k)$ with $k < j$, hence the two factors can be computed by starting with $D(1,1) = A(1,1)$, $L(i,1) = A(i,1) / D(i,1)$ and proceeding iteratively in left-to-right order as follows.

\begin{algorithm}[H]
    \caption{$LDL^T$ factorization}
    \label{alg:ldlt}
    \begin{algorithmic}[1]
        \For{$j = 1, \ldots, n$}
            \State $D(j,j) = A(j,j) - L(j,\ell) \, D(\ell,\ell) \, L^T(\ell,j)$
            \State $L(r,j) = \big( A(r,j) - L(r,\ell) \, D(\ell,\ell) \, L^T(r,j) \big) / D(j,j)$
        \EndFor
    \end{algorithmic}
\end{algorithm}

\Cref{thm:factorization_recursive} can be derived from the following auxiliary result which we will also use in \cref{sec:localization}.

\begin{lemma}\label{thm:factorization}
    In the notation of \cref{rem:selinv_notation}, {we have} that
    \begin{equation}\label{eqn:factorization}
        L(i,j) \, D(j,j)
        =
        A(i,j)
        -
        A(i,\ell) \, A(\ell,\ell)^{-1} \, A(\ell,j)
        .
    \end{equation}
\end{lemma}

\begin{proof}
    The matrices
    \begin{equation}\label{eqn:block_ldlt}
        \hat L := \begin{pmatrix}
            \mathbb{I} &  \\
            A(\bar r,\ell) A(\ell,\ell)^{-1} & \mathbb{I}
        \end{pmatrix}
        , \qquad
        \hat D := \begin{pmatrix}
            A(\ell,\ell) &   \\
                         & S
        \end{pmatrix}
    \end{equation}
    with
    \begin{equation}\label{eqn:schur_complement}
        S := A(\bar r,\bar r) - A(\bar r,\ell) \, A(\ell,\ell)^{-1} \, A(\ell,\bar r)
    \end{equation}
    provide a block $LDL^T$ factorization of $A$ from which the full factorization follows by further factorizing
    \[
        L_\ell D_\ell L_\ell^T := A(\ell,\ell)
        ,\qquad
        L_{\bar r} D_{\bar r} L_{\bar r}^T := S
    \]
    and setting
    \[
        L = \hat L \, \begin{pmatrix}
            L_\ell & \\
                & L_{\bar r}
        \end{pmatrix}
        ,\qquad
        D = \begin{pmatrix}
            D_\ell & \\
                & D_{\bar r}
        \end{pmatrix}
        .
    \]
    We thus compute
    \begin{align}
        L(i,j) \, D(j,j)
        &=
        L_{\bar r}(i,j) \, D_{\bar r}(j,j)
        =
        L_{\bar r}(i,j) \, D_{\bar r}(j,j) \, L_{\bar r}^T(j,j)
        \\&=
        S(i,j)
        =
        A(i,j)
        -
        A(i,\ell) \, A(\ell,\ell)^{-1} \, A(\ell,j)
        ,
    \end{align}
    where we enumerated the rows and columns of $L_{\bar r},D_{\bar r}$ starting from $j$ rather than 1 for consistency with the indexing in the full matrices.
\end{proof}

\begin{proof}[Proof of \cref{thm:factorization_recursive}]
    It follows from the special structure of $L$ and $D$ that
    \begin{align}
        A(i,\ell)
        &=
        L(i,\ell) \, D(\ell,\ell) \, L^T(\ell,\ell)
        ,\\
        A(\ell,\ell)^{-1}
        &=
        L(\ell,\ell)^{-T} \, D(\ell,\ell)^{-1} \, L(\ell,\ell)^{-1}
        ,\\
        A(\ell,j)
        &=
        L(\ell,\ell) \, D(\ell,\ell) \, L^T(\ell,j)
        .
    \end{align}
    The claim follows by inserting these expressions into \cref{eqn:factorization}.
\end{proof}

\subsection{Sparse $LDL^T$ factorization}\label{ssec:sparse_factorization}

It is well known that \cref{alg:ldlt} ($LDL^T$ factorization) applied to a dense matrix $A$ requires $\mathcal{O}(n^3)$ floating-point operations and that these costs may be reduced to the ones shown in \cref{tbl:selinv_costs} if $A$ is sparse.
This subsection briefly recalls some important definitions and results from the theory of sparse factorization which we will use in later sections.
A textbook reference for the material presented here is provided in \cite{Dav06}.

\begin{definition}
    The \emph{graph $G(A) := \big(V(A),E(A)\big)$ of a sparse matrix $A \in \mathbb{C}^{n \times n}$} is {given} by $V(A) := \{1, \ldots, n\}$ and $E(A) := \{(j,i) \mid A(i,j) \neq 0\}$.
\end{definition}

\begin{definition}
    A \emph{fill path} between two vertices $i,j \in V(A)$ is a path $i, k_1, \ldots, k_p, j$ in $G(A)$ such that $k_1, \ldots, k_p < \min\{i,j\}$.
\end{definition}

\begin{theorem}[{{\cite[Theorem 1]{RT78}}}]\label{thm:fill_path}
    In the notation of \cref{rem:selinv_notation} and barring cancellation, {we have} that $(i,j) \in \mathrm{fnz}(A)$ if and only if there is a fill path between $i$ and $j$ in $G(A)$.
\end{theorem}

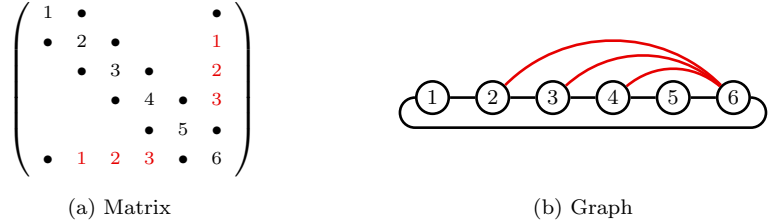
\begin{figure}\centering
    \subfloat[Matrix]{
        \def\l{1em}
        \def\h{0.5em}
        \def\mark#1{#1}
        \begin{tikzpicture}
            \node {$
                \left(\begin{smallmatrix}
                    \hspace{\l} & \hspace{\l} & \hspace{\l} & \hspace{\l} & \hspace{\l} & \hspace{\l} \\
                    \vspace{\h} 1 & \bullet &         &         &         & \bullet \\
                    \vspace{\h} \bullet & 2 & \bullet &         &         & \textcolor{cred}{1} \\
                    \vspace{\h} & \bullet & 3 & \bullet &         & \mark{\textcolor{cred}{2}} \\
                    \vspace{\h} &         & \bullet & 4 & \bullet & \textcolor{cred}{3} \\
                    \vspace{\h} &         &         & \bullet & 5 & \bullet \\
                    \vspace{\h} \bullet & \textcolor{cred}{1} & \mark{\textcolor{cred}{2}} & \textcolor{cred}{3} & \bullet & 6 \\
                \end{smallmatrix}\right)
            $};
        \end{tikzpicture}
    }
    \hspace{4em}
    \subfloat[Graph]{
        \begin{tikzpicture}
            \foreach \i in {1,...,6} {
                \node (\i) [cvertex] at ($0.8*(\i,0)$) {\i};
            }
            \foreach \i [evaluate=\i as \angle using 40+2.5*(4-\i)] in {2,...,4} {
                \draw [cedge,cred] (\i) to [bend left=\angle] (6);
            }
            \foreach \i [remember=\i as \im (initially 1)] in {2,...,6} {
                \draw [cedge] (\im) to  (\i);
            }
            \def\y{0.2}
            \draw [cedge] (1.west) arc (90:270:\y) -- ($(6.east) - 2*(0,\y)$) arc(-90:90:\y);

            \node at (1,-3em) {};
        \end{tikzpicture}
    }
    \caption{
        Illustration of fill-in and level-of-fill for a one-dimensional periodic chain. The black numbers on the diagonal enumerate the vertices, the black dots indicate nonzero off-diagonal elements of the matrix, and the red numbers show the level-of-fill. See \cref{ex:fill_in,def:level_of_fill} for further details.
    }
    \label{fig:fill_in}
\end{figure}

\begin{example}\label{ex:fill_in}
    Consider a matrix with sparsity structure as shown in \cref{fig:fill_in}.
    By \cref{thm:fill_path}, we get fill-in between vertices 4 and 6 because we can connect these two vertices via 3, 2 and 1 which are all numbered less than 4 and 6.
    We do not get fill-in between vertices 3 and 5, however, because all paths between these vertices have to go through either 4 or 6 which are larger than 3.
\end{example}

It follows from \cref{thm:fill_path} that the number of fill-in entries depends not only on the sparsity pattern of $A$ but also on the order of the rows and columns.
While finding an optimal fill-reducing order is an $NP$-hard problem \cite{Yan81}, the following nested dissection algorithm was shown to be asymptotically optimal up to at most a logarithmic factor in \cite{Gil86}.

\begin{algorithm}[H]
    \caption{Nested dissection}
    \label{alg:nested_dissection}
    \begin{algorithmic}[1]
        \State Partition the vertices into three sets $V_1,V_2,V_\mathrm{sep}$ such that there are no edges between $V_1$ and $V_2$.
        \State\label{alg:nested_dissection:conquer}Arrange the vertices in the order $V_1,V_2,V_\mathrm{sep}$, where $V_1$ and $V_2$ are ordered recursively according to the nested dissection algorithm and the order in $V_\mathrm{sep}$ is arbitrary.
    \end{algorithmic}
\end{algorithm}

The rationale for sorting the separator $V_\mathrm{sep}$ last on \cref{alg:nested_dissection:conquer} is that this eliminates all fill paths between $V_1$ and $V_2$ and thus $L(V_2,V_1) = 0$.
However, the submatrix $L(V_\mathrm{sep},V_\mathrm{sep})$ associated with the separator is typically dense; hence the nested dissection order is most effective if $V_\mathrm{sep}$ is small and $V_1,V_2$ are of roughly equal size.

The application of the nested dissection algorithm to a structured 2D mesh is illustrated in \cref{fig:nested_dissection}.
We note that the largest separator $V_\mathrm{sep}$ returned by this algorithm (the blue vertex set in the center of \cref{fig:nested_dissection}) contains $\mathcal{O}\bigl(\sqrt{n}\bigr)$ vertices{;} thus computing the associated dense part $L(V_\mathrm{sep},V_\mathrm{sep})$ alone {requires} $\mathcal{O}\bigl(n^{3/2}\bigr)$ floating-point operations and the full factorization must be at least as expensive to compute.
It was shown in \cite{Geo73} that this lower bound is indeed achieved, which {justifies} the ($d = 2$, Runtime) entry in \cref{tbl:selinv_costs} for the factorization part of the selected inversion algorithm.
The other entries can be derived along similar lines, see e.g.\ \cite{Dav06}, and it has been shown in \cite{Ett19} that the selected inversion step has the same asymptotic complexity as the sparse factorization step (see also \cite{LLYCE09} for a similar but less general result).

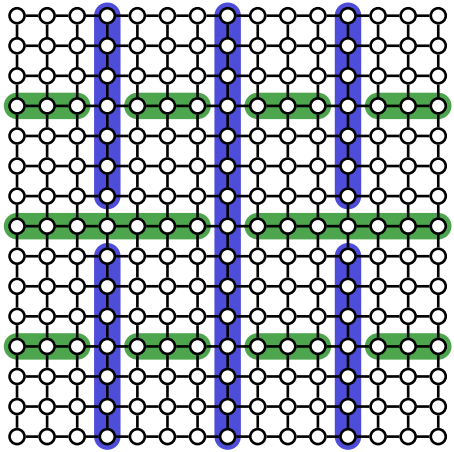
\begin{figure}
    \centering
    \begin{tikzpicture}
        \foreach \i in {1,...,15} {
        \foreach \j in {1,...,15} {
        \node (\i-\j) [cvertex] at ($0.4*(\j,-\i)$) {};
        }
        }
        \foreach \i [remember=\i as \im (initially 1)] in {2,...,15} {
        \foreach \j in {1,...,15} {
        \draw [cedge] (\im-\j) -- (\i-\j) ;
        }
        }
        \foreach \i in {1,...,15} {
        \foreach \j [remember=\j as \jm (initially 1)] in {2,...,15} {
        \draw [cedge] (\i-\jm) -- (\i-\j);
        }
        }

        \def\s{5pt}
        \begin{scope}[on background layer]
            \fill [ cblue!70] \convexpath{1-8,15-8}{\s};

            \fill [cgreen!70] \convexpath{8-1,8-7}{\s};
            \fill [cgreen!70] \convexpath{8-9,8-15}{\s};

            \fill [cblue!70] \convexpath{1-4,7-4}{\s};
            \fill [cblue!70] \convexpath{9-4,15-4}{\s};
            \fill [cblue!70] \convexpath{1-12,7-12}{\s};
            \fill [cblue!70] \convexpath{9-12,15-12}{\s};

            \fill [cgreen!70] \convexpath{4-1,4-3}{\s};
            \fill [cgreen!70] \convexpath{4-5,4-7}{\s};
            \fill [cgreen!70] \convexpath{4-9,4-11}{\s};
            \fill [cgreen!70] \convexpath{4-13,4-15}{\s};
            \fill [cgreen!70] \convexpath{12-1,12-3}{\s};
            \fill [cgreen!70] \convexpath{12-5,12-7}{\s};
            \fill [cgreen!70] \convexpath{12-9,12-11}{\s};
            \fill [cgreen!70] \convexpath{12-13,12-15}{\s};
        \end{scope}
    \end{tikzpicture}
    \caption{
        Nested dissection applied to a structured 2D mesh.
        The vertices marked in blue and green denote alternating separators $V_\mathit{sep}$.
    }
    \label{fig:nested_dissection}
\end{figure}

\subsection{Selected inversion}\label{ssec:selinv}

We now turn our attention towards computing the inverse $A^{-1}$ of a matrix $A$ given its $LDL^T$ factorization. This can be achieved using the following auxiliary result.

\begin{lemma}[\cite{TFC73}]\label{thm:inversion}
    In the notation of \cref{rem:selinv_notation}, {we have} that
    \begin{equation}\label{eqn:inversion}
        A^{-1}(i,j) = D^{-1}(i,j) - A^{-1}(i,r) \, L(r,j)
        .
    \end{equation}
\end{lemma}
\begin{proof}
    The claim follows from $A^{-1} = L^{-T} \, D^{-1} + A^{-1} \, (\mathbb{I} - L)$ which can be verified by {substituting $A^{-1}$ with $L^{-T} D^{-1} L^{-1}$}.
\end{proof}

\Cref{eqn:inversion} has the reverse property of \cref{eqn:factorization}: the right-hand side of \cref{eqn:inversion} depends only on $L,D$ and entries $A^{-1}(i,k)$ with $k > j$; hence the full inverse can be computed by starting with $A^{-1}(n,n) = D(n,n)^{-1}$ and iteratively growing the set of known entries in right-to-left order.
As in the case of the $LDL^T$ factorization, this procedure requires $\mathcal{O}(n^3)$ floating-point operations when applied to a dense matrix but may be reduced to the costs in \cref{tbl:selinv_costs} if $A$ is sparse and only the entries $A^{-1}(i,j)$ with $(i,j) \in \mathrm{fnz}(A)$ are required.
The following algorithm with
\begin{equation}\label{eqn:r_circ}
    r^\circ = r^\circ(j) := \big\{i \in \{j+1, \ldots, {n}\} \mid L(i,j) \neq 0\big\}
\end{equation}
was proposed in \cite{ET75} to achieve this.

\begin{algorithm}[H]
    \caption{Selected inversion}
    \label{alg:selinv}
    \begin{algorithmic}[1]
        \For{$j = n, \ldots, 1$}\label{alg:selinv:loop}
            \State\label{alg:selinv:lwtr}$A^{-1}(r^\circ,j) = -A^{-1}(r^\circ,r^\circ) \, L(r^\circ,j)$
            \State\label{alg:selinv:copy}$A^{-1}(j,r^\circ) = A^{-1}(r^\circ,j)^T$
            \State\label{alg:selinv:diag}$A^{-1}(j,j) = D(j,j)^{-1} - A^{-1}(j,r^\circ) \, L(r^\circ,j)$
        \EndFor
    \end{algorithmic}
\end{algorithm}

\begin{theorem}[\cite{ET75}]\label{thm:selinv_closedness}
    \Cref{alg:selinv} is correct, i.e.\ the computed entries $A^{-1}(i,j)$ agree with those of the exact inverse, and it is closed in the sense that all entries of $A^{-1}$ required at iteration $j$ have been computed in previous iterations $j' > j$.
\end{theorem}
\begin{proof}
    (Correctness.)
    The formulae in \cref{alg:selinv} agree with those of \cref{thm:inversion} except that $r^\circ$ is used instead of $r$ in the products
    $
    A^{-1}(i, r^\circ) \, L(r^\circ, j)
    .
    $
    This does not change the result of the computations since $L(r \setminus r^\circ, j) = 0$ by the definition of $r^\circ$, hence the computed entries are correct.

    (Closedness.)
    The entries of $A^{-1}$ required on \cref{alg:selinv:copy,alg:selinv:diag} are computed on \cref{alg:selinv:lwtr}{;} thus it remains to show that the entries $A^{-1}(i,k)$ with $i,k \in r^\circ(j)$ required on \cref{alg:selinv:lwtr} have been computed in iterations $j' > j$.
    Due to the symmetry of $A$, we can assume $i \geq k$ without loss of generality, and since the diagonal entry $A^{-1}(k,k)$ is explicitly computed on \cref{alg:selinv:diag} in iteration $j = k$, we may further restrict our attention to indices $i > k$.
    Such an entry $A^{-1}(i,k)$ is computed in iteration $j =  k$ if and only if $i \in r^\circ(k)$, hence the claim follows from the following lemma.
\end{proof}

\begin{lemma}[{\cite{ET75}}]\label{thm:selinv_paths}
    In the notation of \cref{rem:selinv_notation} and with $r^\circ(j)$ as in \cref{eqn:r_circ}, {we have} that
    \[
    i,k \in r^\circ(j)
    \quad\text{and}\quad
    i > k
    \qquad\implies\qquad
    i \in r^\circ(k)
    .
    \]
\end{lemma}
\begin{proof}
    According to \cref{thm:fill_path}, $i,k \in r^\circ(j)$ holds if and only if there exist fill paths from $i$ and $k$ to $j$, i.e.\ the graph structure is given by
    \[
    \begin{tikzpicture}
        \foreach \i/\l in {1/j, 2/k, 3/i} {
        \node (\l) [cvertex] at ($1*(\i,0)$) {$\l$};
        }
        \draw [cedge,black] (i) to [bend right=45] (j);
        \draw [cedge,black] (j) to [bend right=42.5] (k);
        \draw [cedge,cred] (i) to [bend left=42.5] (k);
    \end{tikzpicture}
    \vspace{-0.5em}
    \]
    where the two black edges indicate the aforementioned fill paths.
    Concatenating these two paths yields the red fill path from $i$ to $k$; hence the claim follows.
\end{proof}

\section{Exponential localization}\label{sec:localization}

This section will establish in \cref{ssec:localization_ldlt} that the $LDL^T$ factorization computed by the selected inversion algorithm exhibits a localization property similar to that of the density matrix $f_{\beta,E_F}(H)$ described in \cite{Koh96,BBR13}.
This result will be shown as a consequence of the localization of $(H-zI)^{-1}$ described in \cref{ssec:localization_inv} which in turn is related to the rate of convergence of polynomial approximation to $\frac{1}{x - z}$ as discussed in \cref{ssec:polynomial_approximation}.

\subsection{Polynomial approximation}\label{ssec:polynomial_approximation}

We begin by introducing the notation required for characterizing the convergence of polynomial approximation to analytic functions.

\begin{definition}\label{def:asymptotic_rate_of_decay}
    A sequence $a : \mathbb{N} \to [0,\infty)$ is said to {decay exponentially with asymptotic rate $\alpha$} if for all $\tilde \alpha < \alpha$ there exists a constant $C(\tilde \alpha)$ such that $a_k \leq C(\tilde \alpha) \, \exp(-\tilde\alpha k)$ for all $k \in \mathbb{N}$.
    We write
    $
    a_k \lesssim_\varepsilon \exp(-\alpha k)
    $
    for such sequences.
\end{definition}

We note that $a_k \lesssim_\varepsilon \exp(-\alpha k)$ is slightly weaker than $a_k \lesssim \exp(-\alpha k)$ since in the former case the implied prefactor $C(\tilde \alpha)$ is allowed to diverge for $\tilde \alpha \to \alpha$.
For the purposes of this paper, the distinction between $a_k \lesssim_\varepsilon \exp(-\alpha k)$ and $a_k \lesssim \exp(-\alpha k)$ is required for correctness but it is of little practical importance.

\begin{theorem}[{\cite[\S12]{Tre13}, \cite[Thm.\ 4.1]{Saf10}}]\label{thm:polynomial_approximation}
    Let $\mathcal{E} \subset \mathbb{C}$ be a compact and non-polar set and denote by $g_\mathcal{E} : \mathbb{C} \to [0,\infty)$ the Green's function associated with the complement of $\mathcal{E}$ (see \cref{rem:logarithmic_potential_theory} regarding non-polar sets and Green's functions).
    We then have
    \begin{equation}\label{eqn:polynomial_approximation}
        \min_{p \in \mathcal{P}_d} \bigl\|\tfrac{1}{x-z} - p(x)\bigr\|_\mathcal{E}
        \lesssim_\varepsilon
        \exp\bigl(-g_\mathcal{E}(z) \, d \bigr)
    \end{equation}
    where $\mathcal{P}_d$ denotes the set of polynomials with degrees bounded by $d$ and $\|\cdot\|_\mathcal{E}$ denotes the supremum norm on $\mathcal{E}$.
\end{theorem}

\begin{remark}\label{rem:logarithmic_potential_theory}
The notions of non-polar sets and Green's functions in \cref{thm:polynomial_approximation} originate in the field of logarithmic potential theory \cite{Saf10,Ran95}.
Readers unfamiliar with these concepts may take \cref{thm:polynomial_approximation} as the definition of the Green's function $g_\mathcal{E}(z)$: it is the function which maps $z \in \mathbb{C}$ to the asymptotic rate of convergence of polynomial approximation to $\frac{1}{x - z}$ on $\mathcal{E}$.
It is clear that this rate of convergence $g_\mathcal{E}(z)$ is infinity if $\mathcal{E}$ is ``too small'', in which case we refer to $\mathcal{E}$ as polar.
An important example of polar sets are countable sets, and even though there are uncountable polar sets (see e.g.\ \cite[Thm. 5.3.7]{Ran95}), we expect that countable sets are the only polar sets one will encounter in practical applications of the theory presented in this paper.
We therefore encourage the reader to mentally replace ``polar'' with ``countable'' if doing so facilitates the reading of this paper.

The following properties of $g_\mathcal{E}(z)$ are relevant for our purposes.
Unless specified otherwise, these results can be found e.g.\ in \cite{Saf10}.

\begin{itemize}
    \item If $\delta\mathcal{E} \subset \mathbb{C}$ is polar, then $g_\mathcal{E}(z) = g_{\mathcal{E} \cup \delta\mathcal{E}}(z)$, i.e.\ Green's function are invariant under perturbations in polar sets.
    In the context of \cref{thm:localization_inv}, this implies that the localization of $(H-zI)^{-1}$ is independent of the point spectrum of $H$ which confirms similar findings in \cite{OTC19}.

    \item If $\mathcal{E} = [-1,1]$, then
    $
    g_{[-1,1]}(z) = \log\left(\left|z + \sqrt[\star]{z^2-1}\right|\right)
    $
    where
    \[
    \sqrt[\star]{z^2-1}
    :=
    \begin{cases}
        \sqrt{z^2-1} & \text{if } \arg(z) \in (-\tfrac{\pi}{2},\tfrac{\pi}{2}], \\
        -\sqrt{z^2-1} & \text{otherwise}, \\
    \end{cases}
    \]
    and $\sqrt{z}$ denotes the principal square root $\sqrt{z} := \sqrt{|z|} \, \exp\left(i\,\tfrac{\arg(z)}{2}\right)$ with $\arg(z) \in (-\pi,\pi]$.

    \item The Green's function $g_{[a,b]}(z)$ for arbitrary intervals $[a,b]$ can be expressed in terms of the above Green's function as
    \[
    g_{[a,b]}(z) = g_{[-1,1]}\Bigl(\tfrac{2}{b-a} \bigl(z - \tfrac{b+a}{2}\bigr)\Bigr)
    .
    \]

    \item If $\mathcal{E} = [a,b] \cup [c,d]$ with $a < b < c < d$, then
    \[
    g_{\mathcal{E}}(z) = \real\left(\int_a^z f(u) \, (s-u) \, du\right)
    \]
    where
    \[
    f(u) :=
    \frac{1}{ \sqrt{u-a}\,\sqrt{u-b}\,\sqrt{u-c}\,\sqrt{u-d} }
    ,\qquad
    s = \frac{\int_b^c f(u) \, u \, du}{\int_b^c f(u) \, du}
    .
    \]
    This Green's function has been derived in \cite{SSW01}, which also discusses the extension to an arbitrary number of intervals on the real line.
\end{itemize}
\end{remark}

\subsection{Localization of inverse}\label{ssec:localization_inv}

\begin{definition}\label{rem:loc_notation}
    We use the following notation for the remainder of this section.
    \begin{itemize}
        \item $H \in \mathbb{C}^{n \times n}$ denotes a sparse, symmetric matrix, and $i,j \in \{1, \ldots, n\}$ are indices for an entry in the lower triangle, i.e.\ $i \geq j$.
        \item $\mathcal{E} \subset \mathbb{R}$ denotes a compact and non-polar set (cf.\ \cref{rem:logarithmic_potential_theory}) independent of $n$ such that the spectra of all leading submatrices $H(\ell,\ell)$ with $\ell = \{1, \ldots, i\}$ and $i$ ranging from $1$ to $n$ are contained in $\mathcal{E}$.
        We will further comment on this assumption in \cref{rem:submatrix_spectra}.
        \item $z \in \mathbb{C}\setminus\mathcal{E}$ denotes a point outside $\mathcal{E}$.
        \item $L,D$ denote the $LDL^T$ factors of $H-zI$.
        \item $d(i,j)$ denotes the graph distance in $G(H)$, which is defined as the minimal number of edges on any path between $i$ and $j$, or $\infty$ if there are no such paths.
    \end{itemize}
\end{definition}

\begin{theorem}\label{thm:localization_inv}
    In the notation of \cref{rem:loc_notation}, {we have} that
    \[
        \big|(H-zI)^{-1}(i,j)\big|
        \lesssim_\varepsilon
        \exp\bigl(- g_\mathcal{E}(z) \, d(i,j)\bigr)
        .
    \]
    (The notation $\lesssim_\varepsilon$ {was} introduced in \cref{def:asymptotic_rate_of_decay}, and $g_\mathcal{E}(z)$ denotes the Green's function from \cref{thm:polynomial_approximation}.)
\end{theorem}

The proof of \cref{thm:localization_inv} follows immediately from \cref{thm:polynomial_approximation} {and} the following lemma.

\begin{lemma}[{\cite{DMS84}}]\label{polynomial bound}
    In the notation of \cref{rem:loc_notation}, {we have} for all bounded $f : \mathcal{E} \to \mathbb{C}$ that
    \[
        \big|f(H)(i,j)\big|
        \leq
        \inf_{p \in \mathcal{P}_{d(i,j)-1}} \|f - p\|_\mathcal{E}
    \]
    where $\mathcal{P}_{k}$ denotes the space of polynomials of degree $\leq k$.
\end{lemma}

\begin{proof}
    Since $p \in \mathcal{P}_{d(i,j)-1}$, {we have} that $p(H)(i,j) = 0$ and thus
    \begin{align}
        |f(H)(i,j)|
        &\leq
        |p(H)(i,j)|
        +
        |f(H)(i,j) - p(H)(i,j)|
        \\&\leq
        0 +
        \|f(H) - p(H)\|_2
        \\&\leq
        \|f - p\|_\mathcal{E}
        .
    \end{align}
\end{proof}

\subsection{Localization of $LDL^T$ factorization}\label{ssec:localization_ldlt}
Characterizing the localization in the $LDL^T$ factorization requires a new notion of distance introduced in the following definition.

\begin{definition}[{{\cite[\S 10.3.3]{Saa03}}}]\label{def:level_of_fill}
    In the notation of \cref{rem:loc_notation}, the level-of-fill $\mathrm{level}(i,j)$ {is given by}
    \[
        \mathrm{level}(i,j)
        :=
        \max\{0,d_\mathrm{fill}(i,j)-1\}
    \]
    where $d_\mathrm{fill}(i,j)$ denotes the minimal number of edges on any fill path between $i$ and $j$, or $\infty$ if no such path exists.
\end{definition}

An example for the level-of-fill is provided in \cref{fig:fill_in}.

\begin{theorem}\label{thm:localization_ldlt}
    In the notation of \cref{rem:loc_notation}, {we have} that
    \[
        |L(i,j)|
        \lesssim_\varepsilon
        \exp\bigl(- g_\mathcal{E}(z) \, \mathrm{level}(i,j)\bigr)
        .
    \]
\end{theorem}

\begin{proof}
    The claim is trivially true if $\mathrm{level}(i,j) = 0$; hence we restrict ourselves to $i,j$ such that $\mathrm{level}(i,j) > 0$ and $H(i,j) = 0$ in the following.
    According to \cref{thm:factorization}, {we then have} that
    \[
    L(i,j) = -A\bigl(i,\ell^\circ(i)\bigr) \, A_\ell^{-1}\bigl(\ell^\circ(i),\ell^\circ(j)\bigr) \, A\bigl(\ell^\circ(j),j\bigr) / D(j,j)
    \]
    where $A_\ell := A(\ell,\ell)$ with $\ell = \{1, \ldots, j-1\}$ and $\ell^\circ(t) := \big\{ k \in \ell \mid A(t,k) = A(k,t) \neq 0 \big\}$.
    By the definition of level-of-fill, {we have} that $d_\ell(i^\circ,j^\circ) \geq \mathrm{level}(i^\circ,j^\circ)-1$ for all $i^\circ \in \ell^\circ(i)$, $j^\circ \in \ell^\circ(j)$ and with $d_\ell(i,j)$ the graph distance in $G(A_\ell)$;
    thus $A_\ell(i^\circ,j^\circ) \lesssim_\varepsilon \exp\bigl(-g_\mathcal{E}(z)\, \mathrm{level}(i,j)\bigr)$ according to \cref{thm:localization_inv}.
    The claim follows after noting that $\# \ell^\circ(i)$, $\# \ell^\circ(j)$ are bounded independently of $n$ due to the sparsity of $H$,
    and that $|D(j,j)| \geq \min | z - \mathcal{E} |$ since $D(j,j)^{-1} = A_{\bar\ell}^{-1}(j,j)$ with $\bar \ell = \{1, \ldots, j\}$ and the spectrum of $A_{\bar \ell}$ is contained in $\mathcal{E}$ according to the assumptions in \cref{rem:loc_notation}.
\end{proof}

\begin{remark}\label{rem:submatrix_spectra}
    The assumption that the spectra of all leading submatrices $H(\ell,\ell)$ {are} contained in $\mathcal{E}$ in \cref{rem:loc_notation} was introduced specifically to allow for \cref{thm:localization_ldlt}.
    We would like to point out that this assumption can always be satisfied by choosing $\mathcal{E}$ as the convex hull of the spectrum of $H$ and that the rational approximation schemes from \cite{Mou16,LLYE09,Ett19} place the poles away from the real axis and hence outside of this convex hull.
    Furthermore, we expect that the spectra of the submatrices are usually contained in the spectrum of $H$ since in physical terms this corresponds to the assumption that the electronic properties of subsystems agree with those of the overall system.
    If true, the conditions on $\mathcal{E}$ in \cref{rem:loc_notation} are somewhat redundant and we may choose $\mathcal{E}$ simply as a non-polar set containing the spectrum of $H$.
    We will return to this point in \cref{ex:localization_numerics}.
\end{remark}

We conclude from \cref{thm:localization_inv,thm:localization_ldlt} that the entries of both the inverse $(H - zI)^{-1}(i,j)$ {and} the $L$-factor $L(i,j)$ decay exponentially with the same rate $g_\mathcal{E}(z)$ but {with} different notions of distance $d(i,j)$ and $\mathrm{level}(i,j)$, respectively.
This qualitative difference is illustrated in the following example.

\begin{figure}
    \subfloat[$A^{-1}$]{\label{fig:localization_inv}\includegraphics{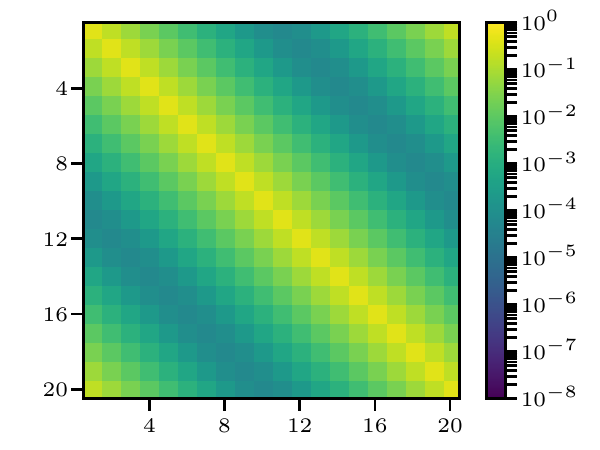}}
    \subfloat[$L$]{\label{fig:localization_ldlt}\includegraphics{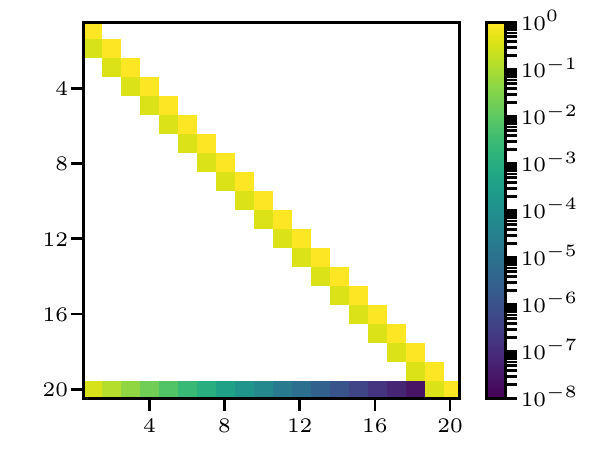}}
    \caption{
        Decay in the inverse and $L$-factor of the matrix from \cref{ex:localization}.
    }
    \label{fig:localization}
\end{figure}

\begin{example}\label{ex:localization}
    Consider the $n \times n$ matrix
    \[
    A(i,j)
    :=
    \begin{cases}
         3 & \text{if } i = j \\
        -1    & \text{if } i = j \pm 1 \mod n, \\
         0    & \text{otherwise} \\
    \end{cases}
    \]
    whose graph structure for $n = 6$ is illustrated in \cref{fig:fill_in}.
    We observe that $\mathrm{level}(n,j) = j-1$ increases monotonically from $j = 1$ to $j = n-2$ and thus $L(n,j)$ decreases monotonically over the same range, see \cref{fig:localization_ldlt}.
    Conversely, $d(n,j) = \min\{j, n-j\}$ has a maximum at $j = \tfrac{n}{2}$ and thus $|A^{-1}(n,j)|$ has a minimum at this value of $j$, see \cref{fig:localization_inv}.
\end{example}

\section{Incomplete selected inversion}\label{sec:iselinv}

\Cref{thm:localization_inv,thm:localization_ldlt} assert that $L(i,j)$ and $A^{-1}(i,j)$ are small if their indices $i$ and $j$ are far apart according to the appropriate notion of distance, which suggests that it should be possible to ignore such entries without losing much accuracy.
Following this idea, we will next present incomplete factorization and selected inversion algorithms (\cref{alg:ildlt,alg:iselinv}) which have been obtained by modifying \cref{alg:ldlt,alg:selinv} such that they operate on only the entries $L(i,j)$ and $A^{-1}(i,j)$ with indices $(i,j)$ from the restricted set
\begin{equation}\label{eqn:ifnz_lof}
    \mathrm{ifnz}(H) := \big\{(i,j) \in \mathrm{fnz}(H) \mid \mathrm{level}(i,j) \leq c\big\}
\end{equation}
for some $c > 0$.
Of course, such a modification is only worth considering if it leads to a substantial improvement in performance and negligible loss of accuracy; hence the main topic of this section will be to assess these two performance metrics.

It is easily seen that restricting the selected inversion algorithm to $\mathrm{ifnz}(H)$ results in linear-scaling computational costs.

\begin{theorem}\label{thm:iselinv_cost}
    The runtime and memory requirements of \cref{alg:ildlt,alg:iselinv} (incomplete factorization and incomplete selected inversion) are given by
    \begin{equation}\label{eqn:iselinv_cost}
        \begin{cases}
            \mathcal{O}(nc) & \text{if } d = 2, \\
            \mathcal{O}(nc^3) & \text{if } d = 3,
        \end{cases}
        \qquad
        \text{and}
        \qquad
        \begin{cases}
            \mathcal{O}(n\log(c)) & \text{if } d = 2, \\
            \mathcal{O}(nc) & \text{if } d = 3,
        \end{cases}
    \end{equation}
    respectively, where $d$ denotes the effective dimension of the problem in the sense of \cref{tbl:selinv_costs} and a nested dissection order is assumed.
\end{theorem}
\begin{proof}
    Nested dissection ordering of $G(A)$ splits the vertices into  $\mathcal{O}\bigl(n c^{-d}\bigr)$ localized clusters $\mathcal{C}_k$ of diameter $\mathcal{O}(c)$ (the white squares in \cref{fig:nested_dissection}) separated by a complementary set $\mathcal{S}$ of
    \begin{align}
        \hspace{6em}&\hspace{-6em}
        \text{(number of $d-1$-dimensional slices)}
        \times
        \text{(number of vertices per slice)}
        =
        \ldots \\ &=
        \mathcal{O}\bigl(n^{1/d} \, c^{-1}\bigr)
        \times
        \mathcal{O}\bigl(n^{(d-1)/d}\bigr)
        =
        \mathcal{O}\bigl(nc^{-1}\bigr)
    \end{align}
    vertices (the colored bars in \cref{fig:nested_dissection}).
    On each of the clusters $\mathcal{C}_k$, incomplete selected inversion proceeds as in the exact algorithm and therefore requires
    \[
    \begin{cases}
        \mathcal{O}\bigl((c^2)^{3/2}\bigr)
        =
        \mathcal{O}\bigl(c^3\bigr)
        &
        \text{if } d = 2, \\
        \mathcal{O}\bigl((c^3)^2\bigr)
        =
        \mathcal{O}\bigl(c^6\bigr)
        &
        \text{if } d = 3 \\
    \end{cases}
    \]
    runtime and
    \[
    \begin{cases}
        \mathcal{O}\bigl(c^2 \, \log(c^2)\bigr)
        =
        \mathcal{O}\bigl(c^2 \, \log(c)\bigr)
        &
        \text{if } d = 2, \\
        \mathcal{O}\bigl((c^3)^{4/3}\bigr)
        =
        \mathcal{O}\bigl(c^4\bigr)
        &
        \text{if } d = 3 \\
    \end{cases}
    \]
    memory.
    Multiplying these estimates by the number of clusters $\mathcal{O}\bigl(nc^{-d}\bigr)$ yields the estimates given in \cref{eqn:iselinv_cost}.
    Each vertex $v$ in the separator set $\mathcal{S}$ is fill-path-connected to $\mathcal{O}\bigl(c^{d-1}\bigr)$ vertices in $\mathcal{S}$ and to
    \[
    \mathcal{O}\left(\sum_{k = 0}^{\log_2(c)} 2^{(d-1) \, k}\right)
    =
    \mathcal{O}\bigl(c^{d-1}\bigr)
    \]
    vertices in the localized clusters $\mathcal{C}_k$ adjacent to $v$,
    where the second estimate follows from the fact that $v$ can be fill-path-connected to at most $V_\mathrm{sep}$ and some subset of either $V_1$ or $V_2$ but not both for each triplet $V_1,V_2,V_\mathrm{sep}$ encountered by the nested dissection algorithm (\cref{alg:nested_dissection}) applied to the localized clusters $\mathcal{C}_k$ adjacent to $v$.
    It follows from closer inspection of \cref{alg:ildlt,alg:iselinv} (incomplete factorization and selected inversion) that the part of the selected inversion algorithm associated with $\mathcal{S}$ requires
    \begin{equation}\label{eqn:separator_runtime}
        \#\mathcal{S}
        \times
        \mathrm{nfpn}^2
        =
        \mathcal{O}\bigl(nc^{-1}\bigr) \times \mathcal{O}\bigl(c^{2(d-1)}\bigr)
        =
        \mathcal{O}\bigl(nc^{2d-3}\bigr)
    \end{equation}
    runtime and
    \begin{equation}\label{eqn:separator_memory}
        \#\mathcal{S}
        \times
        \mathrm{nfpn}
        =
        \mathcal{O}\bigl(nc^{-1}\bigr) \times \mathcal{O}\bigl(c^{d-1}\bigr)
        =
        \mathcal{O}\bigl(nc^{d-2}\bigr),
    \end{equation}
    memory, where $\mathrm{nfpn}$ denotes the number of fill-path-neighbors per $v \in \mathcal{S}$.
    These estimates agree with those given in \cref{eqn:iselinv_cost} up to the logarithmic factor for the memory requirements in the case $d = 2$.
\end{proof}

The following result summarizes our findings regarding the accuracy of incomplete selected inversion.

\begin{theorem}\label{thm:incompleteness_error}
    Assuming \cref{ass:ildlt_entry_bound}, \cref{ass:level_vs_distance} {and} \cref{ass:iselinv_dropped_entries},
    the entries $B(i,j)$ computed by \cref{alg:ildlt,alg:iselinv} (incomplete factorization and incomplete selected inversion) satisfy the bound
    \[
    \big\|(H-zI)^{-1} - B\big\|_{\mathrm{nz}(H)}
    \lesssim_\varepsilon
    \exp\bigl(- 2 \, g_\mathcal{E}(z) \, c\bigr)
    \]
    where $\|A\|_{\mathcal{I}}$ denotes the supremum norm on $\mathcal{I} \subset \{1, \ldots, n\}^2$,
    \[
    \|A\|_{\mathcal{I}}
    :=
    \max_{(i,j) \in \mathcal{I}} |A(i,j)|
    .
    \]
\end{theorem}

Three particularities of \cref{thm:incompleteness_error} deserve further comment.
Firstly, we note \cref{thm:incompleteness_error} predicts a convergence speed which is twice as fast as one might expect based on \cref{thm:localization_inv,thm:localization_ldlt}, namely $2g_\mathcal{E}(z)$ rather than $g_\mathcal{E}(z)$.
This observation can be intuitively explained by noting that the entries $L(i,j)$ dropped by the incomplete algorithm have magnitudes $\lesssim_\varepsilon \exp\bigl(-g_\mathcal{E}(z) \, c\bigr)$ but occur at entries $i,j$ which are at least a distance $\mathrm{level}(i,j) > c$ away from the nonzeros of $H$. Propagating the resulting errors to $\mathrm{nz}(H)$ therefore attenuates them by an additional factor $\exp\bigl(-g_\mathcal{E}(z) \, c\bigr)$ which yields the final error estimate given in \cref{thm:incompleteness_error}.
Secondly, we observe that \cref{thm:incompleteness_error} is based on some conjectures which are plausible and confirmed by numerical evidence but which we have not been able to establish rigorously.
This circumstance somewhat tarnishes our result from a theoretical point of view, but we expect it to have no practical consequences.
Finally, we remark that the assumption referenced in \cref{thm:incompleteness_error} was only introduced to simplify the statement of the result and has little impact regarding the generality of our findings.

The remainder of this section is devoted to the proof of \cref{thm:incompleteness_error}.
Specifically, \cref{ssec:ildlt} will establish in \cref{thm:ildlt_inverse_error_final} a result analogous to \cref{thm:incompleteness_error} for the incomplete factorization step, and \cref{ssec:iselinv} will do the same for the incomplete selected inversion step in \cref{thm:iselinv_final_error}.
\Cref{thm:incompleteness_error} will then follow from \cref{thm:ildlt_inverse_error_final,thm:iselinv_final_error} by a simple application of the triangle inequality.

\begin{definition}\label{rem:iselinv_notation}
    This section follows the notation of \cref{rem:selinv_notation,rem:loc_notation} with $A = H-zI$ as well as the following additions.
    \begin{itemize}
        \item $\tilde r := \{k \in r \mid (k,j) \in \mathrm{ifnz}(H)\}$ with $j$ and $r$ as in \cref{rem:selinv_notation}.
        \item $c$ denotes the cut-off level-of-fill from \cref{eqn:ifnz_lof}.
        \item $\tilde L, \tilde D$ denote the incomplete $LDL^T$ factors and $E$ denotes the dropped entries computed in \cref{alg:ildlt}. Furthermore, we set $\tilde A := \tilde L \tilde D \tilde L^T$.
        \item $B(i,j) \approx A^{-1}(i,j)$ denotes the approximate entries of the inverse and $F$ denotes the dropped entries computed in \cref{alg:iselinv}.
        \item In both \cref{alg:ildlt,alg:iselinv}, we assume that matrix entries which are not specified are set to zero.
    \end{itemize}
\end{definition}

\subsection{Incomplete $LDL^T$ Factorization}\label{ssec:ildlt}

Restricting \cref{alg:ldlt} to $\mathrm{ifnz}(H)$ yields the following incomplete $LDL^T$ factorization (see \cref{rem:iselinv_notation} for notation).

\begin{algorithm}[H]
    \caption{Incomplete $LDL^T$ factorization}
    \label{alg:ildlt}
    \begin{algorithmic}[1]
        \For{$j = 1, \ldots, n$}
            \State\label{alg:ildlt_D}$\tilde D(j,j) = A(j,j) - \tilde L(j,\ell) \, \tilde D(\ell,\ell) \, \tilde L^T(\ell,j)$
            \State\label{alg:ildlt_L}$\tilde L(\tilde r,j) = \big( A(\tilde r,j) - \tilde L(\tilde r,\ell) \, \tilde D(\ell,\ell) \, \tilde L^T(\ell,j) \big) / \tilde D(j,j)$
            \State\label{alg:ildlt_E1}$E(r \setminus \tilde r,j) = \tilde L(r \setminus \tilde r,\ell) \, \tilde D(\ell,\ell) \, \tilde L^T(\ell,j)$
            \State\label{alg:ildlt_E2}$E(j,r \setminus \tilde r) = E(r \setminus \tilde r, j)^T$
        \EndFor
    \end{algorithmic}
\end{algorithm}

\Cref{alg:ildlt} is precisely the symmetric version of the well-known incomplete LU factorization commonly used as a preconditioner in iterative methods for linear systems, see e.g.\ \cite[\S 10.3]{Saa03}.
Keeping track of the dropped entries $E$ {in} \cref{alg:ildlt_E1,alg:ildlt_E2} is not required in an actual implementation, but doing so in \cref{alg:ildlt} allows us to conveniently formulate the following results regarding the errors introduced by restricting the sparsity pattern of $L$.

\begin{theorem}[{{\cite[Proposition 10.4]{Saa03}}}]\label{thm:ildlt_error}
    In the notation of \cref{rem:iselinv_notation}, {we have} that $\tilde L\tilde D\tilde L^T = A+E$.
\end{theorem}
\begin{proof}
    We note that $\mathrm{level}(i,j) > 0$ for all $i \in r \setminus \tilde r$; hence $A(r \setminus \tilde r,j) = 0$ and
    \[
    A(r\setminus\tilde r,j) + E(r\setminus\tilde r,j) - \tilde L(r\setminus\tilde r,\ell) \, \tilde D(\ell,\ell) \, \tilde L^T(\ell,j)
    =
    0
    \]
    according to \cref{alg:ildlt_E1} of \cref{alg:ildlt}.
    {Since} $E(\tilde r,j) = 0$ and $\tilde L(r \setminus \tilde r,j) = 0$, we can thus combine \cref{alg:ildlt_L,alg:ildlt_E1} to
    \begin{equation}\label{eqn:ildlt_error_L_recursion}
        \tilde L(r,j)
        =
        \big( A(r,j) + E(r,j) - \tilde L(r,\ell) \, \tilde D(\ell,\ell) \, \tilde L^T(r,j) \big) / \tilde D(j,j)
        ,
    \end{equation}
    and similarly we can rewrite \cref{alg:ildlt_D} as
    \begin{equation}\label{eqn:ildlt_error_D_recursion}
        \tilde D(j,j) = A(j,j) + E(j,j) - \tilde L(j,\ell) \, \tilde D(\ell,\ell) \, \tilde L^T(\ell,j)
        .
    \end{equation}
    The claim follows after noting that \cref{eqn:ildlt_error_L_recursion,eqn:ildlt_error_D_recursion} are precisely the recursion formulae of the exact $LDL^T$ factorization in \cref{alg:ldlt} applied to the symmetric matrix $A + E$.
\end{proof}

\begin{theorem}\label{thm:ildlt_inverse_error}
    In the notation of \cref{rem:iselinv_notation} and assuming $\|E\|_2 < \delta := \min |z - \mathcal{E}|$, {we have} that
    \begin{equation}\label{eqn:ildlt_inverse_error}
        \begin{aligned}
            \hspace{4em}&\hspace{-4em}
            \big|\big((H-zI)^{-1} - (H+E-zI)^{-1}\big)(i,j)\big|
            \lesssim_\varepsilon \ldots \\ & \lesssim_\varepsilon
            \sum_{\tilde{\imath}, \tilde{\jmath} = 1}^n
            \exp\Bigl(- g_\mathcal{E}(z) \, \big(d(i,\tilde{\imath}) + d(\tilde{\jmath}, j)\big)\Bigr) \, \big| E(\tilde{\imath},\tilde{\jmath}) \big|
            +
            \frac{\delta^{-2} \, \|E\|_2^2}{\delta - \|E\|_2}
            .
        \end{aligned}
    \end{equation}
    This bound is illustrated in \cref{fig:ildlt_error}.
\end{theorem}
\begin{proof}
    Expanding $(H+E-zI)^{-1}$ in a Neumann series around $H-zI$, we obtain
    \begin{equation}\label{eqn:neumann_expansion}
        \begin{aligned}
            \hspace{4em}&\hspace{-4em}
            (H-zI)^{-1} - (H+E-zI)^{-1}
            = \ldots \\ &=
            (H-zI)^{-1} \, E \, (H-zI)^{-1}
            -
            \sum_{k = 2}^\infty \big(- (H-zI)^{-1} \, E\big)^k \, (H-zI)^{-1}
            .
        \end{aligned}
    \end{equation}
    The claim follows by estimating the entries of $(H-zI)^{-1}$ in the first term using \cref{thm:localization_inv} and bounding the entries of the second term through its operator norm.
\end{proof}

\begin{figure}
    \subfloat[$|A^{-1} - \tilde A^{-1}|$]{\label{fig:ildlt_error}\includegraphics{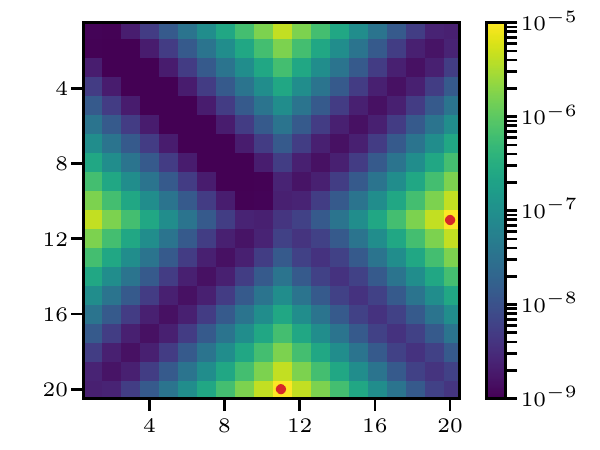}}
    \subfloat[$|\tilde A^{-1} - B|$]{\label{fig:iselinv_error}\includegraphics{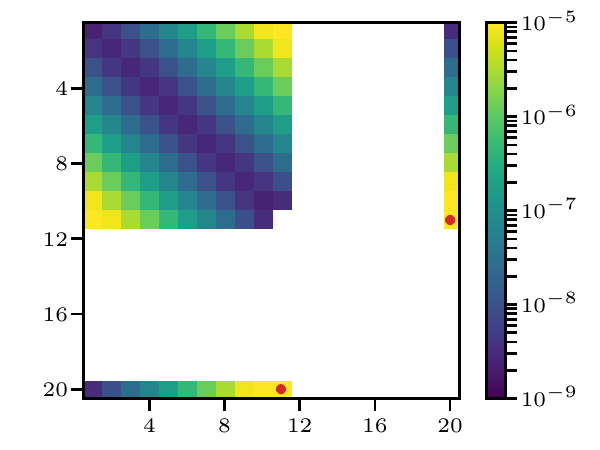}}
    \caption{
    Error introduced by the incomplete selected factorization step (left) and incomplete selected inversion step (right) applied to the matrix from \cref{ex:localization} with a cut-off level-of-fill $c = 9$.
    The red dots mark the nonzero entries in $E$ and $F$, respectively.
    }
    \label{fig:incompleteness_errors}
\end{figure}

\Cref{thm:ildlt_inverse_error} provides an a-posteriori error estimate for the inverse $(H+E-zI)^{-1}$ in terms of the dropped entries $E$, which could be used in an adaptive truncation scheme where $\mathrm{ifnz}(H)$ is of the form
\begin{equation}\label{eqn:ifnz_tol}
\mathrm{ifnz}(H)
=
\big\{
(i,j) \in \mathrm{fnz}(H)
\mid
|\tilde L(i,j)| \geq \tau
\big\}
\end{equation}
for some tolerance $\tau \geq 0$, see \cite[\S 10.4]{Saa03}.
Conversely to the level-of-fill-based scheme from \cref{eqn:ifnz_lof}, such a tolerance-based scheme would control the error but not the amount of fill-in since the perturbed entries $|\tilde L(i,j)|$ may fail to be small even when the corresponding entries $L(i,j)$ are {small}. Both schemes thus require further information about the perturbed factor $\tilde L$ or equivalently about the dropped entries $E(i,j)$ in order to simultaneously control the accuracy {and} the computational effort.
Specifically, in the case of the level-of-fill scheme \cref{eqn:ifnz_lof} we need to understand
\begin{itemize}
    \item the sparsity pattern of $E$ since this impacts the number of terms and the size of the exponential factor in \cref{eqn:ildlt_inverse_error}, and
    \item the magnitudes of the nonzero entries $E(\tilde \imath,\tilde \jmath)$.
\end{itemize}
The first of these two points is easily addressed.

\begin{lemma}\label{thm:E_level}
    In the notation of \cref{rem:iselinv_notation}, {we have} that
    \[
    E(i,j) \neq 0
    \quad\implies\quad
    c < \mathrm{level}(i,j) \leq 2c+1
    .
    \]
    In particular, the number of nonzero entries per row or column of $E$ is bounded independently of $n$.
\end{lemma}

The proof of this result will make use of the following lemma.

\begin{lemma}\label{thm:ildlt_entry_position}
    In the notation of \cref{rem:iselinv_notation} and barring cancellation, {we have} that $E(i,j) \neq 0$ if and only if $(i,j) \in \mathrm{fnz}(H) \setminus \mathrm{ifnz}(H)$ and there exists a $k \in \ell$ such that $(i,k), (j,k) \in \mathrm{ifnz}(H)$.
\end{lemma}
\begin{proof}
    The claim follows by noting that \cref{alg:ildlt_E1} in \cref{alg:ildlt} performs a nonzero update on $E(i,j)$ if and only if
    \[
    i \in r \setminus \tilde r
    \quad\iff\quad
    (i,j) \in \mathrm{fnz}(H) \setminus \mathrm{ifnz}(H)
    \]
    and there exists a $k \in \ell$ such that
    \begin{equation}
        \tilde L(i,k) \, \tilde L(j,k) \neq 0
        \quad\iff\quad
        (i,k), (j,k) \in \mathrm{ifnz}(H)
        .
    \end{equation}
\end{proof}

\begin{proof}[Proof of \cref{thm:E_level}]
    According to \cref{thm:ildlt_entry_position}, {we have} that
    \[
    E(i,j) \neq 0
    \implies
    (i,j) \in \mathrm{fnz}(H) \setminus \mathrm{ifnz}(H)
    \implies
    \mathrm{level}(i,j) > c
    .
    \]
    To derive the upper bound on $\mathrm{level}(i,j)$, let us assume that $E(i,j) \neq 0$. Then, \cref{thm:ildlt_entry_position} guarantees that there exists a vertex $k \in \ell$ such that $i,k$ and $j,k$ are connected by fill paths of lengths at most $c+1$ (recall from \cref{def:level_of_fill} that the level-of-fill is the length of the shortest path minus 1).
    Concatenating these two paths yields a fill path between $i$ and $j$ of length at most $2c+2$; hence $\mathrm{level}(i,j) \leq 2c+1$.
    Finally, the claim regarding the sparsity of $E$ follows from the fact that there are at most $\mathcal{O}\bigl(c^d\bigr)$ vertices $i$ within a distance $2c+1$ from a fixed vertex $j$ in a graph associated with a problem of effective dimension $d$.
\end{proof}

Regarding the second of the above points, namely the magnitudes of the dropped entries $E(\tilde \imath,\tilde \jmath)$, we expect that the following result holds.

\begin{conjecture}\label{ass:ildlt_entry_bound}
    In the notation of \cref{rem:iselinv_notation}, {we have} that
    \[
    |E(i,j)| \lesssim_\varepsilon \exp\bigl(-g_\mathcal{E}(z) \, \mathrm{level}(i,j)\bigr)
    .
    \]
\end{conjecture}

\begin{proof}[Discussion]
    Reversing the substitutions in the proof of \cref{thm:factorization_recursive}, we obtain
    \[
    E(i,j)
    =
    \tilde L(i,\ell) \, \tilde D(\ell,\ell) \, \tilde L^T(\ell,j)
    =
    \tilde A(i,\ell) \, \tilde A(\ell,\ell)^{-1} \, \tilde A(\ell,j)
    ,
    \]
    and expanding the latter formula to first order in $\|E\|_2$ as in \cref{thm:ildlt_inverse_error} yields
    \begin{equation}\label{eqn:ildlt_entry_expansion}
        \begin{aligned}
            E(i,j)
            &=
            A(i,\ell) \, A(\ell,\ell)^{-1} \, A(\ell,j)
            \ldots\\&\qquad + E(i,\ell) \, A(\ell,\ell)^{-1} \, A(\ell,j)
            \ldots\\&\qquad - A(i,\ell) \, A(\ell,\ell)^{-1} E(\ell,\ell) \, A(\ell,\ell)^{-1} \, A(\ell,j)
            \ldots\\&\qquad + A(i,\ell) \, A(\ell,\ell)^{-1} \, E(\ell,j)
            + \mathcal{O}\bigl(\|E\|^2_2\bigr)
            .
        \end{aligned}
    \end{equation}
    \Cref{thm:localization_ldlt} guarantees that the first term
    \[
    A(i,\ell) \, A(\ell,\ell)^{-1} \, A(\ell,j)
    =
    L(i,j) \, D(j,j)
    \]
    on the right-hand side of \cref{eqn:ildlt_entry_expansion} satisfies $\big|L(i,j) \, D(j,j)\big| \lesssim_\varepsilon \exp\bigl(-g_\mathcal{E}(z) \, c\bigr)$, but bounding the remaining terms is challenging because the magnitudes of these terms recursively depend on the errors committed earlier.

    To illustrate this point, let us assume we have a bound
    $
    |E(\tilde \imath, \tilde \jmath)|
    \leq
    C_0
    $
    with
    $
    C_0
    \sim \exp\bigl( - g_{\mathcal{E}}(z) \, c\bigr)
    $
    for all entries of $E$ on the right-hand side of \cref{eqn:ildlt_entry_expansion} such that we can bound e.g.\ the second term by
    \begin{equation}\label{eqn:sum_term_bound}
        \big|E(i,\ell) \, A(\ell,\ell)^{-1} \, A(\ell,j)\big|
        \leq
        C_0 \,
        \sum_{k \in \ell} \big|A_\ell^{-1}(k,\ell) \, A(\ell,j)\big|
    \end{equation}
    where $A_\ell := A(\ell,\ell)$.
    From the sparsity of $A(\ell,j)$ and the localization of $A_\ell^{-1}$, it then follows that the sum on the right-hand side of \cref{eqn:sum_term_bound} decays exponentially for an appropriate ordering of the terms and can therefore be bounded by some constant $C$ independent of $n$. In general, this constant $C$ will be larger than one, however, since some $k \in \ell$ may well be close to $j$ in terms of the graph distance on $G(A_\ell)$ such that the corresponding terms are not small.
    Bounding the other terms in \cref{eqn:ildlt_entry_expansion} similarly, we thus obtain $|E(i,j)| \leq C \, C_0$ for some constant $C > 1$.

    For the next entry $E(i',j')$ to be estimated using \cref{eqn:ildlt_entry_expansion}, the entry $E(i,j)$ that we just estimated may now appear on the right-hand side such that we have to assume the bound
    $
    |E(\tilde \imath', \tilde \jmath')|
    \leq
    C \, C_0
    $
    for these entries. Proceeding analogously as above, we then obtain the bound $|E(i',j')| \leq C^2 \, C_0$ which is worse by a factor {of} $C > 1$ than the bound in the preceding step and worse by a factor {of} $C^2 > 1$ than the bound two steps ago.
    We therefore conclude that any estimate on the dropped entries $E(i,j)$ deteriorates exponentially with every recursive application of \cref{eqn:ildlt_entry_expansion}.

    The key issue in the above analysis is that without further knowledge about the entries $E(i,j)$, we have to assume that all the error terms in the recursion formula \cref{eqn:ildlt_entry_expansion} accumulate rather than cancel.
    We conjecture that such an accumulation of errors cannot occur, at least for matrices which are ``well-behaved'' in a suitable sense, but a rigorous proof of this claim requires deeper insight into the structure of the incomplete $LDL^T$ factorization and is left for future work.
\end{proof}

\Cref{ass:ildlt_entry_bound} suggests that the incomplete factorization exhibits the same localization as the exact factorization, which is in principle enough to derive an a-priori bound from  the a-posteriori bound in \cref{thm:ildlt_inverse_error}.
However, we introduce one more assumption in order to simplify the final result.

\begin{assumption}\label{ass:level_vs_distance}
    Either $\mathrm{level}(i,j) = \infty$ or $\mathrm{level}(i,j) \leq C\, d(i,j)$ for some $C > 0$ independent of $i,j$ and $n$.
\end{assumption}

\begin{proof}[Discussion]
    We have seen in \cref{ex:localization} that this assumption is not satisfied in the case of one-dimensional periodic chains, but we expect that this counterexample is the proverbial exception which proves the rule.
    In particular, we conjecture that \cref{ass:level_vs_distance} is always satisfied in dimensions $d > 1$ and if a nested dissection order is used, since in this case every pair of vertices is connected by many paths and it seems unlikely that the nested dissection order would place a high-numbered vertex on all the short paths.
    This hypothesis is supported by our numerical experiments presented in \cref{ex:iselinv_convergence} below.
    Furthermore, we will see in \cref{ex:level_vs_distance} that even if \cref{ass:level_vs_distance} is violated, the conclusions that we draw from it still hold up to some minor modifications.
\end{proof}

\begin{corollary}\label{thm:ildlt_inverse_error_final}
    In the notation of \cref{rem:iselinv_notation}{,} and assuming \cref{ass:ildlt_entry_bound} {and} \cref{ass:level_vs_distance}, {we have} that
    \[
    \big\|(H-zI)^{-1} - (H+E-zI)^{-1}\big\|_{\mathrm{nz}(H)}
    \lesssim_\varepsilon
    \exp\bigl(- 2 \, g_\mathcal{E}(z) \, c\bigr)
    .
    \]
\end{corollary}

\begin{proof}
    It follows from \cref{thm:E_level} (sparsity of $E$) and \cref{ass:ildlt_entry_bound} (localization of $E$) that for all $\tilde\imath,\tilde\jmath \in \{1, \ldots, n\}$ {we have} that
    $
        |E(\tilde\imath,\tilde\jmath)|
        \lesssim_\varepsilon
        \exp\bigl(-g_\mathcal{E}(z) \, c\bigr)
        .
    $
    Inserting this estimate into the bound from \cref{thm:ildlt_inverse_error} (a-posteriori error bound) yields
    \begin{equation}\label{eqn:ildlt_entry_bound_with_estimates}
        \begin{aligned}
            \hspace{4em}&\hspace{-4em}
            \big|\big((H-zI)^{-1} - (H+E-zI)^{-1}\big)(i,j)\big|
            \lesssim_\varepsilon \ldots \\ & \lesssim_\varepsilon
            \sum_{(\tilde{\imath}, \tilde{\jmath}) \in \mathrm{nz}(E)}
            \exp\Bigl(- g_\mathcal{E}(z) \, \big(d(i,\tilde{\imath}) + d(\tilde{\jmath}, j) + c\big)\Bigr)
            +
            \frac{\delta^{-2} \, \|E\|_2^2}{\delta - \|E\|_2}
            .
        \end{aligned}
    \end{equation}
    We are only interested in entries $(i,j) \in \mathrm{nz}(H)$ for which $d(i,j) \leq 1${;} thus we conclude from the triangle inequality and \cref{ass:level_vs_distance} ($d(i,j) \gtrsim \mathrm{level}(i,j)$) that for all $(\tilde\imath,\tilde\jmath) \in \mathrm{nz}(E)$ {we have} that
    \[
    d(i,\tilde\imath) + d(\tilde\jmath,j) + 1
    \geq
    d(\tilde\imath,\tilde\jmath)
    \gtrsim
    \mathrm{level}(\tilde\imath,\tilde\jmath)
    \geq
    c + 1
    .
    \]
    In particular, we note that if $d(i,\tilde\imath) \lesssim \frac{c}{2}$, then $d(\tilde\jmath,j) \gtrsim \frac{c}{2}$ and vice versa, which allows us to bound the first term in \cref{eqn:ildlt_entry_bound_with_estimates} by
    \begin{align}
        \hspace{4em}&\hspace{-4em}
        \sum_{(\tilde{\imath}, \tilde{\jmath}) \in \mathrm{nz}(E)}
        \exp\Bigl(- g_\mathcal{E}(z) \, \big(d(i,\tilde{\imath}) + d(\tilde{\jmath}, j) + c\big)\Bigr)
        \lesssim\ldots\\&\lesssim
        2\sum_{\substack{\tilde \imath \text{ such that} \\ d(i,\tilde\imath) \lesssim \frac{c}{2}}} \,\,
        \sum_{\substack{\tilde \jmath \text{ such that} \\ d(\tilde\jmath,j) \gtrsim c - d(i,\tilde\imath)}}
        \exp\Bigl(- g_\mathcal{E}(z) \, \big(d(i,\tilde{\imath}) + d(\tilde{\jmath}, j) + c\big)\Bigr)
        + \ldots \\ &\qquad\qquad +
        \sum_{\substack{\tilde \imath \text{ such that} \\ d(i,\tilde\imath) \gtrsim \frac{c}{2}}} \,\,
        \sum_{\substack{\tilde \jmath \text{ such that} \\ d(\tilde\jmath,j) \gtrsim \frac{c}{2}}}
        \exp\Bigl(- g_\mathcal{E}(z) \, \big(d(i,\tilde{\imath}) + d(\tilde{\jmath}, j) + c\big)\Bigr)
        \\
        &\lesssim_\varepsilon
        \exp\bigl(-2\,g_\mathcal{E}(z) \, c\bigr)
    \end{align}
    where on the last line we estimated the infinite sums using the boundedness of the geometric series and the finite sum over $\tilde\imath$ was estimated as the largest term in the sum times the bounded number of terms.

    The second term in \cref{eqn:ildlt_entry_bound_with_estimates} can be bounded using Gershgorin's circle theorem and the facts that all diagonal entries satisfy $E(i,i) = 0$, all off-diagonal entries satisfy $|E(i,j)| \lesssim_\varepsilon \exp\bigl(-g_\mathcal{E}(z) \, c\bigr)$, and the number of off-diagonal entries is bounded independently of $n$,
    which yields $\|E\|_2 \lesssim_\varepsilon \exp\big(-g_\mathcal{E}(z) \, c\bigr)$.
    The claim then follows by combining the bounds on the two terms in \cref{eqn:ildlt_entry_bound_with_estimates}.
\end{proof}

\subsection{Incomplete Selected Inversion}\label{ssec:iselinv}

Restricting \cref{alg:selinv} to $\mathrm{ifnz}(H)$ yields the following incomplete selected inversion algorithm (see \cref{rem:iselinv_notation} for notation).

\begin{algorithm}[H]
    \caption{Incomplete selected inversion}
    \label{alg:iselinv}
    \begin{algorithmic}[1]
        \For{$j = n, \ldots, 1$}
            \State\label{alg:iselinv:lwtr}$B(\tilde r,j) = -B(\tilde r,\tilde r) \, \tilde L(\tilde r,j)$
            \State\label{alg:iselinv:copy}$B(j,\tilde r) =  B(\tilde r,j)^T$
            \State\label{alg:iselinv:diag}$B(j,j) = \tilde D(j,j)^{-1} - B(j,\tilde r) \, \tilde L(\tilde r,j)$
            \State\label{alg:iselinv:F1}$F(r\setminus\tilde r,j) = B(r\setminus\tilde r, \tilde r) \, \tilde L(\tilde r,j)$
            \State\label{alg:iselinv:F2}$F(j,r\setminus\tilde r) = F(r\setminus\tilde r, j)^T$
            \State\label{alg:iselinv:F3}$F(j,j) = F(j,\tilde r) \, \tilde L(\tilde r, j)$
        \EndFor
    \end{algorithmic}
\end{algorithm}

As in \cref{alg:ildlt}, keeping track of the dropped entries $F$ on \cref{alg:iselinv:F1,alg:iselinv:F2,alg:iselinv:F3} is not required in an actual implementation but doing so facilitates our discussion of the errors committed by this algorithm.

Our error analysis for \cref{alg:iselinv} proceeds along the same lines as in \cref{ssec:ildlt}:
we first establish an a-posteriori bound in terms of the dropped entries $F$ in \cref{thm:iselinv_aposteriori}, then we argue that $|F(i,j)|$ should decay like $|A^{-1}(i,j)|$ in \cref{ass:iselinv_dropped_entries}, and finally we derive an a-priori bound based on this conjecture in \cref{thm:iselinv_final_error}.
For all of these steps, we will need the following result which establishes that $\tilde A^{-1} = (A + E)^{-1}$ exhibits the same localization as $A^{-1}$.

\begin{lemma}\label{thm:ildlt_inverse_localization}
    In the notation of \cref{rem:iselinv_notation} and assuming \cref{ass:ildlt_entry_bound}, {we have} that
    \[
    |\tilde A^{-1}(i,j)|
    \lesssim_\varepsilon
    \exp\bigl(-g_\mathcal{E}(z) \, d(i,j)\bigr)
    .
    \]
\end{lemma}
\begin{proof}
    According to \cref{thm:ildlt_inverse_error} (a-posteriori error bound for $\tilde A^{-1}$), \cref{thm:E_level} (sparsity of $E$) and \cref{ass:ildlt_entry_bound} (localization of $E$), {we have} that
    \begin{align}
        |\tilde A^{-1}(i,j)|
        &\lesssim_\varepsilon
        |A^{-1}(i,j)|
        +
        \sum_{(\tilde\imath,\tilde\jmath) \in \mathrm{nz}(E)}
        \exp\Bigl(
        -g_\mathcal{E}(z) \, \big(
        d(i,\tilde\imath) +
        c +
        d(\tilde\jmath,j)
        \big)
        \Bigr)
        .
    \end{align}
    The claim follows by estimating the first term on the right hand side using \cref{thm:localization_inv} (localization of $A^{-1}$) and the second term using the boundedness of geometric series.
\end{proof}

\begin{theorem}\label{thm:iselinv_aposteriori}
    In the notation of \cref{rem:iselinv_notation} and assuming \cref{ass:ildlt_entry_bound}, {we have} that
    \begin{equation}\label{eqn:selinv_error}
        \begin{aligned}
            \Big|\big(\tilde A^{-1} - B\big)(i,j)\Big|
            &\lesssim_\varepsilon
            \sum_{\tilde{\jmath} \, = j+1}^n
            \exp\bigl(-g_\mathcal{E}(z) \, \mathrm{level}(\tilde{\jmath},j )\bigr) \, |F(i,\tilde{\jmath})|
            +\ldots \\&\qquad +
            \sum_{\tilde{\imath}, \tilde{\jmath} \, = j+1}^n
            \exp\Bigl(-g_\mathcal{E}(z) \, \big(\mathrm{level}(i,\tilde{\imath}) + \mathrm{level}(\tilde{\jmath},j)\big)\Bigr) \, |F(\tilde{\imath},\tilde{\jmath})|
            .
        \end{aligned}
    \end{equation}
    This bound is illustrated in \cref{fig:iselinv_error}.
\end{theorem}

\begin{proof}
    Let us first consider the application of \cref{alg:selinv} (exact selected inversion) to the matrix $\tilde A = A + E = \tilde L \tilde D \tilde L^T$.
    We note that the entries $\tilde A^{-1}(i',j')$ computed by this algorithm after iteration $j$ depend only {on} $\tilde L(\cdot, \ell)$, $\tilde D(\ell,\ell)$ {and} $\tilde A^{-1}(\bar r,\bar r)$,
    hence iterations $j' = j-1, \ldots, 1$ may be interpreted as a map $\phi : \big(\tilde L(\cdot,\ell), \tilde D(\ell,\ell), \tilde A^{-1}(\bar r, \bar r)\bigr) \mapsto \tilde A^{-1}$ which must be unique since the map from $\tilde A^{-1}$ to $\big(\tilde L(\cdot,\ell), \tilde D(\ell,\ell), \tilde A^{-1}(\bar r, \bar r)\bigr)$ is injective.
    This uniqueness allows us to determine $\phi$ by applying selected inversion to the block-$LDL^T$ factorization from \cref{eqn:block_ldlt}, which yields
    \begin{equation}\label{eqn:phi_formula}
        \tilde A^{-1}
        =
        \left(\begin{smallmatrix}
            \tilde A(\ell,\ell)^{-1} +
            \tilde A(\ell,\ell)^{-1} \tilde A(\ell,\bar r) \, \tilde A^{-1}(\bar r,\bar r) \, \tilde A(\bar r, \ell) \, \tilde A(\ell,\ell)^{-1}
            \quad&
            -\tilde A(\ell,\ell)^{-1} \tilde A(\ell, \bar r) \, \tilde A^{-1}(\bar r,\bar r)
            \\[0.5em]
            \phantom{\tilde A(\ell,\ell)^{-1} \tilde A(\ell,\ell)^{-1} \tilde A(\ell,\bar r) \, } -\tilde A^{-1}(\bar r,\bar r) \, \tilde A(\bar r, \ell) \, \tilde A(\ell,\ell)^{-1}
            \quad&
            \phantom{-\tilde A(\ell,\ell)^{-1} \tilde A(\ell,\bar r) \, } \tilde A^{-1}(\bar r,\bar r)
            \\
        \end{smallmatrix}\right)
        .
    \end{equation}
    Note that this is indeed a map in terms of $\tilde L(\cdot, \ell)$, $\tilde D(\ell,\ell)$ since all of the submatrices in \cref{eqn:phi_formula} other than $\tilde A(\bar r,\bar r)^{-1}$ can be computed from $\tilde L(\cdot, \ell)$, $\tilde D(\ell,\ell)$.

    Let us now assume for the moment that \cref{alg:iselinv} (incomplete selected inversion) only drops entries in $B\bigl(\bar r, j\bigr)$ and $B\bigl(j, \bar r\bigr)$ such that%
    \footnote{
    We would like to emphasize that this simple formula only holds for the first iteration $j$ where entries {are} dropped, since in later iterations $j' < j$ the error introduced by the dropped entries may propagate into other entries of $B$.
    }
    \[
    B(\bar r,\bar r) = \tilde A^{-1}(\bar r, \bar r) + F(\bar r, \bar r)
    .
    \]
    Since by assumption the incomplete inversion does not perform any additional mistakes after iteration $j$, {we have} that $B = \phi\bigl(B(\bar r, \bar r)\bigr)$ where for brevity we dropped the arguments of $\phi$ other than $\tilde A^{-1}(\bar r, \bar r)$, and since $\phi$ is affine in $\tilde A^{-1}(\bar r, \bar r)$ it further follows that
    \begin{equation}\label{eqn:iselinv_phi_linearity}
        \begin{array}{r @{{}={}} c @{{}+{}} l}
            B
            =
            \phi\bigl(B(\bar r, \bar r)\bigr)
            &
            \phi\bigl(\tilde A^{-1}(\bar r, \bar r)\bigr)
            &
            \phi\bigl(F(\bar r, \bar r)\bigr)
            -
            \phi(0)
            \\&
            \tilde A^{-1}
            &
            \phi\bigl(F(\bar r, \bar r)\bigr)
            -
            \phi(0)
            .
        \end{array}
    \end{equation}
    In the simplified case where errors occur only in $B\bigl(\bar r, j\bigr)$ and $B\bigl(j, \bar r\bigr)$, the claim then follows by estimating the entries of $\phi\bigl(F(\bar r, \bar r)\bigr) - \phi(0)$ using the localization of $\tilde A(\ell,\ell)^{-1}$ described in  \cref{thm:ildlt_inverse_localization}.
    The general estimate follows by applying \cref{eqn:iselinv_phi_linearity} repeatedly for each $j$.
\end{proof}

\begin{conjecture}\label{ass:iselinv_dropped_entries}
    In the notation of \cref{rem:iselinv_notation}, {we have} that
    \[
    |F(i,j)| \lesssim_\varepsilon \exp\bigl(-g_\mathcal{E}(z) \, d(i,j)\bigr)
    .
    \]
\end{conjecture}

\begin{proof}[Discussion]
    From the proof of \cref{thm:iselinv_aposteriori}, it follows that $F$ can be computed recursively according to
    \begin{equation}\label{eqn:F_recursion}
        F(i,j)
        =
        \tilde A^{-1}(i,j)
        -
        F\bigl(i,r(i)\bigr) \, M^T\bigl(r(i),j\bigr)
        +
        M\bigl(i,r(j)\bigr) \, F\bigl(r(j),r(j)\bigr) \, M^T\bigl(r(j),j\bigr)
    \end{equation}
    where
    \[
    M(i,\tilde{\imath})
    =
    \begin{cases}
        A_{\ell(\tilde{\imath})}^{-1}\bigl(i, \ell(\tilde{\imath})\bigr) \, A\bigr(\ell(\tilde{\imath}), \tilde{\imath}\bigr) &  i < \tilde{\imath} \\
        0 & \text{otherwise}
    \end{cases}
    \]
    and $A_\ell := A(\ell,\ell)$.
    Proving \cref{ass:iselinv_dropped_entries} thus faces the same obstacle as \cref{ass:ildlt_entry_bound}, namely that the errors committed at iteration $j$ depend on errors committed at previous iterations $j' > j$ such that any bound deteriorates exponentially in the recursion depth.
\end{proof}

\begin{corollary}\label{thm:iselinv_final_error}
    In the notation of \cref{rem:iselinv_notation}{,} and assuming \cref{ass:ildlt_entry_bound}, \cref{ass:level_vs_distance} {and} \cref{ass:iselinv_dropped_entries}, {we have} that
    \[
    \big\|(H+E-zI)^{-1} - B^{-1}\big\|_{\mathrm{nz}(H)}
    \lesssim_\varepsilon
    \exp\bigl(- 2 \, g_\mathcal{E}(z) \, c\bigr)
    .
    \]
\end{corollary}

\begin{proof}
    Analogous to \cref{thm:ildlt_inverse_error_final}.
\end{proof}

\section{Numerical experiments}\label{sec:experiments}

This section illustrates the theory presented in this paper at the example of a toy Hamiltonian $H \in \mathbb{R}^{n \times n}$ with entries $H(i,j)$ given by
\begin{equation}\label{eqn:toy_hamiltonian}
H(i,j)
:=
\begin{cases}
    (-1)^{d(i,1)} & \text{if } i = j, \\
    -\tfrac{1}{2 d} & \text{if } i \sim j, \\
    0    & \text{otherwise}, \\
\end{cases}
\end{equation}
where $d \in \{1,2,3\}$ denotes the dimension and $i \sim j$ if $i$ and $j$ are nearest neighbors in a $d$-dimensional Cartesian mesh with periodic boundary conditions.
We note that the off-diagonal entries $H(i,j) = -\frac{1}{2 d}$ correspond to a shifted and scaled finite-difference discretization of the $d$-dimensional Laplace operator, and the diagonal entries $H(i,i) = (-1)^{d(i,1)}$ take the form of a chequerboard pattern where each vertex has the opposite sign compared to its neighbors.
In two and three dimensions, we use the nested dissection order illustrated in \cref{fig:nested_dissection} to improve the sparsity of the $LDL^T$ factorization, while in one dimension we use a simple left-to-right order as shown in \cref{fig:fill_in}.
All numerical experiments have been performed on a single core of an Intel Core i7-8550 CPU (1.8 GHz base frequency, 4 GHz turbo boost) using the Julia programming language \cite{BEKS17}.
A Julia package developed as part of this work is available at \url{https://github.com/ettersi/IncompleteSelectedInversion.jl}.

\begin{figure}
    \centering
    \subfloat[Factorization]{\includegraphics{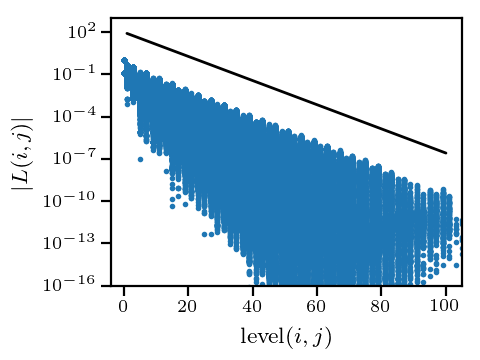}}
    \subfloat[Inverse]{\includegraphics{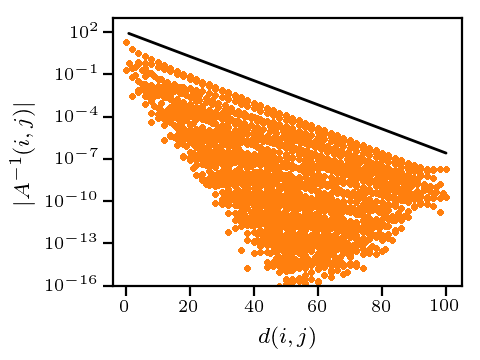}}
    \caption{
    Localization of the $L$-factor (a) and inverse (b) of the matrix $H-0.98I$ with $H$ constructed according to \cref{eqn:toy_hamiltonian} with $d = 2$ and $\sqrt{n} = 100$.
    The black lines indicate the rate of decay $g_\mathcal{E}(z)$ predicted by \cref{thm:localization_inv,thm:localization_ldlt}.
    }
    \label{fig:localization_rates}
\end{figure}

\begin{example}[localization]\label{ex:localization_numerics}
The chequerboard pattern along the diagonal causes the spectrum of $H$ to cluster in the two intervals $\mathcal{E} := [-\sqrt{2},-1] \cup [1, \sqrt{2}]$ such that one may think of $H$ as the Hamiltonian matrix of an insulator with band gap $(-1,1)$.
\Cref{fig:localization_rates} compares the entries of the $LDL^T$ factorization and inverse of $H-zI$ against the predictions of \cref{thm:localization_inv,thm:localization_ldlt}, and we observe that the theory is matched perfectly.
In particular, the entries of the $L$-factor decay with the same rate $g_\mathcal{E}(z)$ as the inverse, which indicates that the spectr{a} of all leading submatrices of $H$ are indeed contained in $\mathcal{E}$ as conjectured in \cref{rem:submatrix_spectra}.

The excellent agreement between the theoretical and empirical convergence rates is a consequence of the simple sparsity pattern in \cref{eqn:toy_hamiltonian}, and the agreement may be worse for a more realistic Hamiltonian.
\end{example}

\begin{example}[scaling and convergence]\label{ex:iselinv_convergence}
\Cref{fig:scaling} numerically confirms the $\mathcal{O}(nc)$ ($d = 2$) and $\mathcal{O}(nc^3)$ ($d = 3$) scaling of incomplete selected inversion predicted by \cref{thm:iselinv_cost}.
Furthermore, \cref{fig:convergence} demonstrates that incomplete selected inversion indeed converges exponentially in the cut-off level-of-fill $c$ with a rate of convergence equal to twice the localization rate $g_\mathcal{E}(z)$ as predicted in \cref{thm:incompleteness_error}.
\end{example}

\begin{figure}
    \centering
    \subfloat[Scaling in $n$ for $c = 4$]{\label{fig:nscaling}\includegraphics{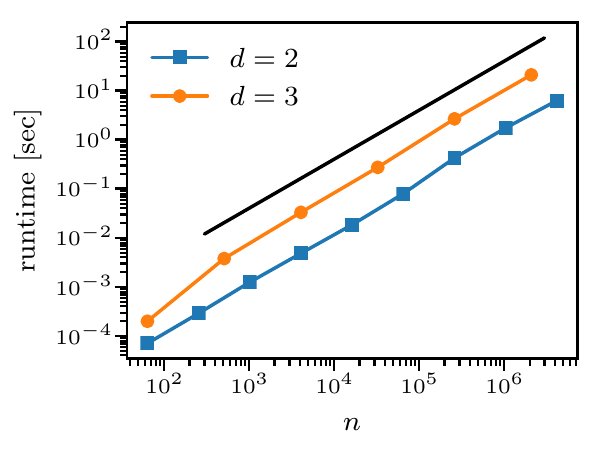}}
    \subfloat[Scaling in $c$ for $n^{1/d} = 32$]{\label{fig:cscaling}\includegraphics{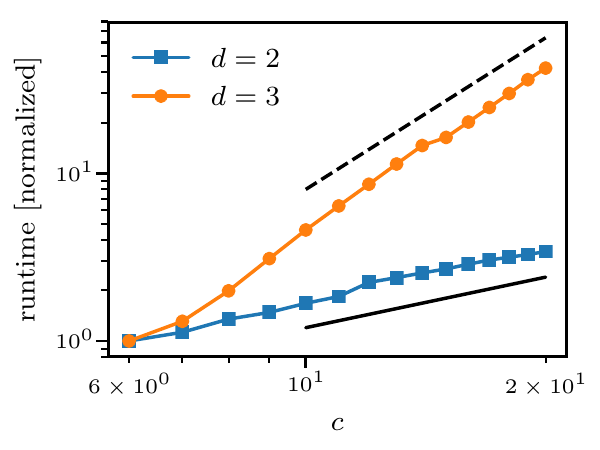}}
    \caption{
    Scaling of the incomplete selected inversion algorithm with respect to the matrix size $n$ (left) and cut-off level-of-fill $c$ (right).
    The black lines indicate $\mathcal{O}(n)$ scaling (Figure (a)), $\mathcal{O}(c)$ scaling (Figure (b), solid) and $\mathcal{O}(c^3)$ scaling (Figure (b), dashed) as predicted by \cref{thm:iselinv_cost}.
    The reported runtimes are the minimum out of three runs of selected inversion applied to the matrix $H$ from \cref{eqn:toy_hamiltonian}.
    In Figure (b), the runtimes have been scaled by $1/\text{runtime}(c = 6)$ to ensure comparable y-values for the two plots.
    }
    \label{fig:scaling}
\end{figure}

\begin{figure}
    \centering
    \subfloat[$d = 2, z = 0.98$]{\label{fig:convergence_2d}\includegraphics{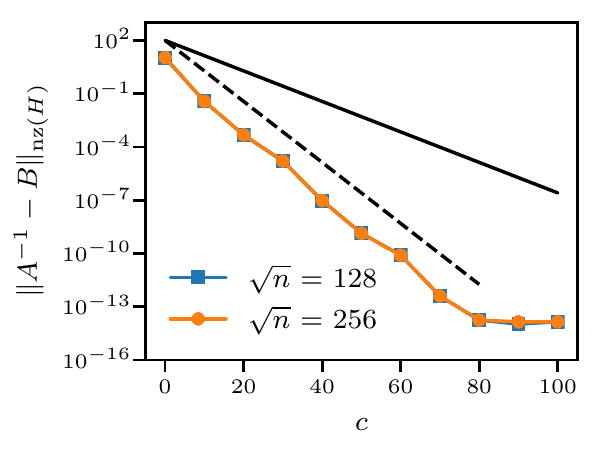}}
    \subfloat[$d = 3, z = 0$]{\label{fig:convergence_3d}\includegraphics{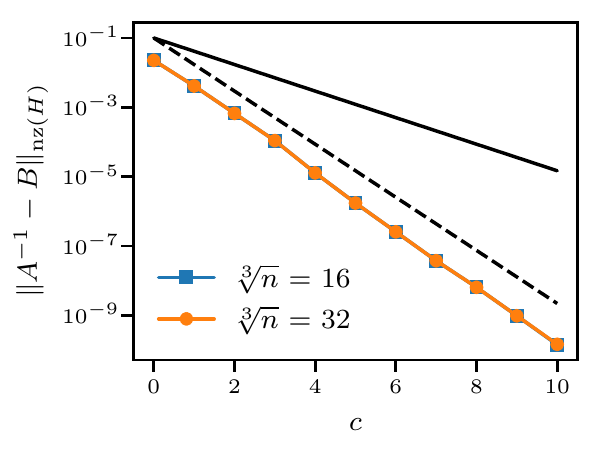}}
    \caption{
    Error vs.\ cut-off level-of-fill $c$ of incomplete factorization and selected inversion algorithm applied to $A = H-zI$ with $H$ as in \cref{eqn:toy_hamiltonian}.
    The solid black lines indicate exponential decay with rate $g_\mathcal{E}(z)$, and the dashed lines indicate twice this rate.
    }
    \label{fig:convergence}
\end{figure}

\begin{figure}
    \centering
    \subfloat[Localization]{\label{fig:periodic_1D_localization}\includegraphics{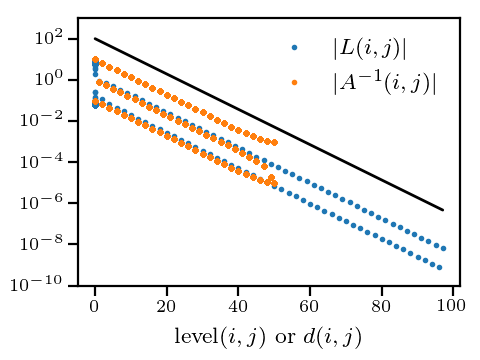}}
    \subfloat[Error]{\label{fig:periodic_1D_error}\includegraphics{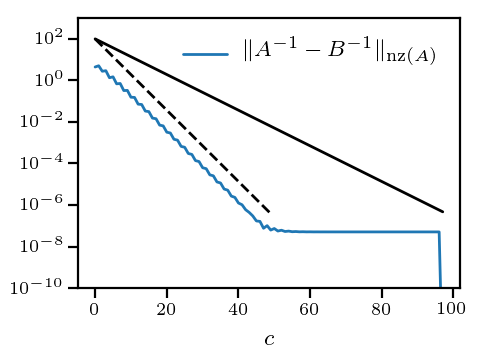}}
    \caption{
    Localization (left) and convergence of incomplete factorization and selected inversion (right) of the matrix $H-zI$ with $H$ as in \cref{eqn:toy_hamiltonian}, $d = 1$ $n = 100$ and $z = 0.98$.
    The solid black lines indicate exponential decay with rate $g_\mathcal{E}(z)$, and the dashed line indicates twice this rate.
    }
    \label{fig:periodic_1D}
\end{figure}

\begin{example}\label{ex:level_vs_distance}
Finally, we continue the discussion of \cref{ass:level_vs_distance} ($d(i,j) \gtrsim \mathrm{level}(i,j)$) which was used to derive \cref{thm:incompleteness_error} but which we have seen to be violated in the case of one-dimensional periodic chains.
\Cref{fig:periodic_1D} shows that even in this case, the incomplete algorithm converges with rate $2 g_\mathcal{E}(z)$ as predicted by \cref{thm:incompleteness_error}, but the convergence breaks down at a cut-off level-of-fill $c$ of about $\frac{n}{2}$ after which the error stagnates.

In the framework of \cref{sec:iselinv}, this observation may be explained as follows.
Due to the simple graph structure of $H$, the matrix of dropped entries $E$ from \cref{alg:ildlt} contains precisely two nonzero entries at locations $i = n$, $j = c+2$ and the transpose thereof,
and by \cref{ass:ildlt_entry_bound} these entries satisfy
\[
|E(n,c+2)|
\lesssim_\varepsilon
\exp\bigl(-g_\mathcal{E}(z) \, \mathrm{level}(n,c+2)\bigr)
=
\exp\bigl(-g_\mathcal{E}(z) \, (c+1)\bigr)
.
\]
(In this case, \cref{ass:ildlt_entry_bound} can easily be proven since the incomplete factorization algorithm only drops a single entry.)
According to \cref{thm:ildlt_inverse_error}, the error due to the incomplete factorization is thus upper-bounded by
\begin{align}
    \|A^{-1} - \tilde A^{-1}\|_{\mathrm{nz}(A)}
    &\lesssim_\varepsilon
    \max_{(i,j) \in \mathrm{nz}(A)}
    \exp\Bigl(-g_\mathcal{E}(z) \, \big( d(i,n) + c+1 + d(c+2,j) \big) \Bigr)
    \\&\lesssim_\varepsilon
    \exp\Bigl(-g_\mathcal{E}(z) \, \big( c+1 + \min \{ d(c+2,1), d(c+2,n-1)\} \big) \Bigr)
    \\&=
    \begin{cases}
        \exp\bigl(-2\,g_\mathcal{E}(z) \, (c+1)\bigr) & \text{if } c \leq \tfrac{n-4}{2}, \\
        \exp\bigl(-g_\mathcal{E}(z) \, (n-2)\bigr) & \text{otherwise}, \\
    \end{cases}
    \label{eqn:onedimensional_error_analysis}
\end{align}
and a similar bound can be derived for the error $\|\tilde A^{-1} - B\|_{\mathrm{nz}(A)}$ introduced by the incomplete selected inversion step.
We note that \cref{eqn:onedimensional_error_analysis} describes precisely the behaviour observed in \cref{fig:periodic_1D_error}.

As discussed after \cref{ass:level_vs_distance}, we expect that $d(i,j) \gtrsim \mathrm{level}(i,j)$ is rarely violated in dimensions $d > 2$ and if a nested dissection vertex order is used.
This example further demonstrates that even if \cref{ass:level_vs_distance} is violated, the incomplete selected inversion algorithm still converges at the rate $2 g_\mathcal{E}(z)$ until the cut-off level-of-fill $c$ becomes $\mathcal{O}(n)$, at which point the speedup of the incomplete selected inversion compared to the exact algorithm vanishes anyway.
\end{example}

\section{Conclusion}\label{sec:conclusion}

We have shown that the $LDL^T$ factorization of a sparse and well-conditioned\footnote{
We call a matrix $A = H -zI$ well-conditioned if $z \notin \mathcal{E}$ with $\mathcal{E}$ the ``smoothed'' spectrum of $H$ described in \cref{rem:loc_notation}.
} matrix $A$ exhibits a localization property similar to that of $A^{-1}$, and we have developed algorithms which exploit this property to compute selected entries of $A^{-1}$ in $\mathcal{O}(n)$ runtime and memory.
This opens up a new class of linear-scaling electronic structure algorithms which we expect to be highly competitive compared to existing algorithms due to reasons which we shall explain next.

Most linear-scaling electronic structure codes in use today proceed according to the following outline \cite{Goe99,BM12}.
\begin{itemize}
    \item Express the density matrix $f_{\beta,E_F}(H)$ as the minimizer of some functional $\mathcal{F} : \mathbb{R}^{n \times n} \to \mathbb{R}$,
    \[
    f_{\beta,E_F}(H)
    =
    \argmin_{\Gamma \in \mathbb{R}^{n \times n}} \mathcal{F}(\Gamma)
    .
    \]
    \item Compute an approximation $\tilde \Gamma \approx f_{\beta,E_F}(H)$ with $\mathcal{O}(n)$ effort by minimizing $\mathcal{F}(\tilde \Gamma)$ over the space of matrices with a certain prescribed sparsity pattern.
\end{itemize}
For simplicity of our argument, we will assume that the functional $\mathcal{F}(\Gamma)$ is given by the McWeeny purification function \cite{McW60}
\[
\mathcal{F}(\Gamma) := \trace\Bigl(\big(3\Gamma^2 - 2\Gamma^3\big) \, (H - \mu I)\Bigr)
\]
and the minimization is performed using the conjugate gradients algorithm, but we expect that similar observations also hold for more sophisticated linear-scaling schemes.
The bulk of the compute time is then spent on evaluating products of sparse matrices, each of which requires $\mathcal{O}\bigl(n c^{2d}\bigr)$ floating-point operations\footnote{
The number of floating-point operations required by a sparse matrix product $C = AB$ can be estimated as follows: each of the $n$ rows of $C$ has $\mathcal{O}(c^d)$ nonzero entries, each of these entries is computed by taking the inner product of a row of $A$ and a column of $B$, and both of these vectors have $\mathcal{O}(c^d)$ nonzero entries.
} assuming an effective dimension $d$ and that both matrix arguments have localization length $c$.
Furthermore, the number of conjugate gradient steps and hence the number of matrix products required to reach a fixed accuracy scales with the inverse square root of the band gap in $H$ \cite[\S III.D]{Goe99}.
In contrast, the number of selected inversions required by the PEXSI scheme scales only logarithmically with the band gap and temperature \cite{LLYE09,Mou16,Ett19}, and each selected inversion with a localization length $c$ requires $\mathcal{O}\bigl(nc\bigr)$ operations for two-dimensional problems and $\mathcal{O}\bigl(nc^3\bigr)$ operations for three-dimensional problems; see \cref{thm:iselinv_cost}.
The incomplete PEXSI scheme thus makes order of magnitudes fewer calls to the low-level routine (selected inversion / matrix product) than the optimization scheme, and each such call runs faster for large enough $c$.
Combining these two factors leads us to expect that our iPEXSI scheme will outperform existing linear-scaling algorithms at least in the regime of large enough localization lengths $c$.

As a case in point, we report that the selected inversion algorithm applied to the matrix \cref{eqn:toy_hamiltonian} with $d = 2$, $\sqrt{n} = 256$ and $c = 20$ takes 0.4 seconds, while evaluating the $20$th power of the same matrix $H$ takes 7.5 seconds (see \cref{sec:experiments} for details regarding hardware and software).
Evaluating a power of $H$ with similar localization properties as those assumed in the selected inversion is roughly 20 times slower for these particular parameters, and this ratio will tip even further in favor of the selected inversion algorithm as we increase the localization length $c$.

What is needed next in order to realize the promised advantage of the iPEXSI algorithm is a massively parallel high-performance implementation of the incomplete selected inversion algorithm comparable to that presented for the exact algorithm in \cite{JLY16}.
Developing such a code will be the topic of future work, but we would like to point out that the parallelization strategies from \cite{JLY16} also apply to the incomplete factorization and selected inversion algorithms and hence we expect similar parallel scaling.
Closely related work regarding the parallel implementation of the incomplete LU factorization with arbitrary level-of-fills (as opposed to the more wide-spread ILU(0) and ILU(1) factorizations) can be found in \cite{KK97,HP01,SZW03,DC11}.
Furthermore, an alternative ILU algorithm based on iterative refinement of a trial factorization and designed specifically to improve parallel scaling has recently been proposed in \cite{CP15}.
While this algorithm was found to be highly effective at finding factorizations suitable for preconditioning, it is unclear whether it is applicable in the context of the selected inversion algorithm where the accuracy requirements are much more stringent. Additionally, the algorithm from \cite{CP15} will only lead to an asymptotic speedup for the selected inversion algorithm if a similarly parallelizable algorithm for the selected inversion step can be developed.
It is not obvious whether such an algorithm exists, since the algorithm from \cite{CP15} is based on \cref{thm:ildlt_error} which has no analogue for the selected inversion step.

\section*{Acknowledgements}
The author would like to thank Christoph Ortner (thesis supervisor), Nick Trefethen (thesis referee) and Andreas Dedner (thesis referee) for valuable input.

\bibliographystyle{siamalpha}
\bibliography{/home/ettersi/Documents/library/library.bib}

\end{document}